\documentclass{scrartcl}

\usepackage[algo2e,vlined,linesnumbered,ruled,resetcount]{algorithm2e}

\usepackage{bbm} 
\usepackage{dsfont} 
\usepackage{amsmath, amssymb}
\usepackage{amsthm}
\usepackage{mathabx} 
\usepackage{mathrsfs}
\usepackage{url}
\usepackage{amssymb,amsfonts}
\usepackage{color}
\usepackage{shuffle}
\usepackage[linecolor=white,backgroundcolor=white,bordercolor=white,textsize=tiny]{todonotes}
\usepackage{imakeidx}
\usepackage{faktor}
\makeindex[name=general]
\usepackage{bm}
\usepackage{enumitem}
\setlist[enumerate]{label=\roman*.}
\allowdisplaybreaks
\usepackage{breqn}
\usepackage{silence}
\usepackage[flushleft]{threeparttable}
\usepackage{longtable}
\usepackage{enumitem}
\usepackage{hyperref}
\hypersetup{colorlinks}
\usepackage{cleveref}

\usepackage{enumitem}
\setlist[enumerate]{leftmargin=.5in}
\setlist[itemize]{leftmargin=.5in}

\newtheorem{theorem}{Theorem}
\newtheorem{corollary}[theorem]{Corollary}
\newtheorem{example}[theorem]{Example}
\newtheorem{lemma}[theorem]{Lemma}

\newtheorem{proposition}[theorem]{Proposition}
\newtheorem{remark}[theorem]{Remark}
\newtheorem{definition}[theorem]{Definition}

\crefname{hypothesis}{Hypothesis}{Hypotheses}
\crefname{algocf}{alg.}{algs.}
\Crefname{algocf}{Algorithm}{Algorithms}
\crefname{algocfline}{alg.}{algs.}
\Crefname{algocfline}{Algorithm}{Algorithms}

\usepackage{amsopn}

 \newtheorem{conjecture}[theorem]{Conjecture}

\usepackage{listings}

\newcommand{\R}{\mathbb{R}}

\newcommand{\N}{\mathbb{N}}

\newcommand{\E}{\mathbb{E}}

\newcommand{\argmin}{\operatorname{argmin}}

\newcommand{\ad}{\operatorname{ad}}

\newcommand{\GL}{\operatorname{GL}}

\newcommand{\mute}[1]{}

\let\todon\todo
\renewcommand{\todo}[1]{\todon{\color{red}#1}}

\newcommand\w[1]{{\color{cyan}\mathbf{#1}}}

\newcommand\e{\mathsf{e}} 
\newcommand\LieAlg{\mathfrak{g}}
\newcommand\LieAlgSymb{\mathfrak{g}^{\mathsf{symb}}}
\newcommand\DEF[1]{\textbf{#1}}
\newcommand\desc{\mathsf{desc}}

\newcommand\IIS{\mathsf{IIS}}
\newcommand\G{\mathcal{G}}
\newcommand\BCH{\mathsf{BCH}}
\newcommand\rev{\mathsf{rev}}
\newcommand\DX{\mathsf{D}_X}
\newcommand\aBCH{\mathsf{aBCH}}
\newcommand\TIME{\tau} 

\newlength\mylen

\ifpdf
\hypersetup{
  pdftitle={The barycenter in free nilpotent Lie groups and its application to iterated-integrals signatures},
  pdfauthor={M. Clausel, J. Diehl, R. Mignot, L. Schmitz, N. Sugiura, and K. Usevich}
}
\fi

\title{The barycenter in free nilpotent Lie groups and its application to iterated-integrals signatures}

\author{Marianne Clausel\thanks{Universit\'{e} de Lorraine, CNRS, IECL, F-54000 Nancy, France}\and Joscha Diehl\thanks{University of Greifswald, Walther-Rathenau-Str. 47, 17489 Greifswald, Germany}\and Raphael Mignot\footnotemark[1]\and\break Leonard Schmitz\footnotemark[2]\and Nozomi Sugiura\thanks{Japan Agency for Marine-Earth Science and Technology, 237-0061 Yokosuka, Japan}\and Konstantin Usevich\thanks{Universit\'{e} de Lorraine, CNRS, CRAN, F-54000 Nancy, France}}

\begin{document}
\maketitle

\begin{abstract}
  We establish the well-definedness of the barycenter (in the sense of Buser and Karcher)
  for every integrable measure on
  the free nilpotent Lie group of step $L$ (over $\mathbb{R}^d$).
  We provide two algorithms for computing it, using methods
  from Lie theory (namely, the Baker-Campbell-Hausdorff formula)
  and from the theory of Gröbner bases of modules.
  Our main motivation stems from measures induced by iterated-integrals signatures,
  and we calculate the barycenter for the signature
  of the Brownian motion.
\end{abstract}
\noindent\textbf{Keywords:}
    barycenter, group mean, iterated-integrals signature, tensor algebra, BCH formula, Brownian motion, Gröbner bases of modules

\noindent\textbf{MSC codes:}
60L10, 
22E25, 
60J65, 
13P10, 
15A69  

\noindent\textbf{Funding:} This work was supported by a Trilateral ANR-DFG-JST AI program. In particular, MC, RM and KU were supported by ANR Project ANR-20-IADJ-0003. NS was supported by JST Project JPMJCR20G5. JD and LS were supported by DFG Project 442590525.


\newpage

\tableofcontents



\newpage

\section{Introduction}

\noindent{\textbf{Iterated integrals}}\\
The iterated-integrals signature of a (smooth enough) continuous curve
$X : [0,\TIME] \to \R^d$, $X = (X^{(1)}, \dots, X^{(d)})$,
is a collection of real numbers, indexed by words over the alphabet
$\{\w{1},\ldots,\w{d}\}$, whose entry corresponding to
a word $w = w_1 \dots w_k$ is given by
\begin{align}
  \label{eq:integrals}
  \Big\langle \mathrm{IIS}(X)_{0,\TIME}, w \Big\rangle 
  &= \int_{0 \le t_1  \le \cdots \le t_k \leq \TIME}\mathrm dX_{t_1}^{(w_1)} \dots \mathrm dX_{t_k}^{(w_k)} \\
  &= \int_{0 \le t_1  \le \cdots \le t_k \leq \TIME} \dot X_{t_1}^{(w_1)} \dots \dot X_{t_k}^{(w_k)}\,\mathrm dt_1 \dots \mathrm dt_k.\notag
\end{align}
Its definition dates back to topological work in the 1950s \cite{bib:Che1957}.
It appeared in the context of controlled differential equations in the 1970s \cite{bib:Fli1981}
and has been used since the 1990s in the theory of \emph{rough paths}
for a pathwise approach to stochastic analysis \cite{bib:Lyo1998}.

For a random path, i.e. a stochastic process,
the concept of \emph{expected signature} \cite{bib:Faw2002,bib:Ni2012}
has proven to be useful.
For example, it often characterizes the law of the stochastic process
\cite{bib:CL2016,bib:BDMN2021},
a fact that has been successfully applied
in data science \cite{bib:SO2021,bib:CO2022,bib:Sug2021}.
Maybe surprisingly, the expected signature
of several classes of stochastic processes 
is available either in closed form
or can be computed by solving a fixed-point equation
\cite{bib:LN2015,bib:FHT2022}.


\bigskip
\noindent{\textbf{Signatures and Lie groups}}\\
The iterated-integrals signature 
of a curve
is \emph{grouplike} and hence
takes values
in a nonlinear space
(it lives in a Lie group when considering
the truncated signature).
Expectation in the aforementioned works
is taken in the tensor algebra,
which is the linear ambient space.
It is a well-known fact (and indeed essential
for the cited uniqueness results)
that this expectation is \emph{not} grouplike
anymore (except for a Dirac measure).

\bigskip
\noindent{\textbf{Lie group barycenters}}\\
In the present work, we consider a different kind
of expectation, called the \emph{Lie group mean}
or \emph{Lie group barycenter},
which \emph{does} still live in the group.
In the context of signatures,
the group mean 
can thus be interpreted
as the \emph{signature of some path}.%
\footnote{
  Practically constructing a path
  from a group element is the ``reconstruction problem''.
  It is highly non-trivial, but possible in low dimension, e.g. \cite{bib:PSS2019,bib:AFS2019}. 
}


The concept of a Lie group barycenter
goes back to
\cite{bib:BK1981}.
See  \cite{bib:PL2020} for a recent
exposition that puts it
in context with other concepts
of geometric means.
Barycenters in the group of rotations are treated in \cite{bib:M2002}. 


\bigskip
\noindent{\textbf{Contributions}}\\
Our contributions are as follows:
\begin{itemize}

    \item We establish unique existence
    of the barycenter for arbitrary integrable
    measures on grouplike elements,
    \Cref{thm:main}.
    Previous results needed the assumption
    of compact support, which is not necessary
    in the ``free'' case.
    
    \item We show that, for discrete measure (e.g., a collection of  grouplike ``samples''), there is a finite time algorithm to compute the barycenter (unlike existing approaches relying on iterative methods, such as fixed-point iteration \cite{bib:PA2012}). For this, compare
    \Cref{th:groupmean} and \Cref{prop:1},
    together with implementations in \verb|SageMath| and \verb|python|.
\footnote{
\url{https://github.com/diehlj/free-nilpotent-lie-group-barycenter}}
    
    \item We explicitly calculate 
    the barycenter for the iterated-integrals
    signature of the Brownian motion.
    
\end{itemize}

\bigskip
\noindent{\textbf{Notation}}\newline
We use the following variable names throughout.
\begin{itemize}
    \item 
    $d$ dimension of multivariate time series
    \item $N$ number of time series we want to compute the mean of 
    \item $\TIME$ (maximal) length of time series 
    \item $L$ truncation level (of the free Lie algebra or the tensor algebra)
\end{itemize}

\section{Background}\label{sec:background}

The following definitions and results can be found, for example, in
\cite[Chapter 7]{bib:FV2010}. A recent exposition, with a notation similar to the
one used here, can be found in \cite[Section 2]{bib:DPRT2022}.

\subsection{Free Lie algebras and iterated-integral signatures}
For a fixed dimension $d$, we consider the alphabet $A := \{\w{1},\ldots,\w{d}\}$
and its Kleene closure $A^*$, i.e. the words on $A$ (including the empty word $\e$).
The \DEF{length} of a word $w$ is denoted by $|w|$.
The tensor algebra over $\R^d$ can be realized as the $\R$-vector space
over $A^*$,
\begin{align*}
  T(\R^d) &:= \bigoplus_{k=0}^\infty (\R^d)^{\otimes k} \cong \operatorname{span}_\R (A^*).
\end{align*}
Elements of it are finite, formal sums of words
\begin{align}
  \label{eq:formal}
  \sum_{w \in A^*} c_w\ w,
\end{align}
where the coefficients $c_w \in \R$ are zero, for all but finitely many words
$w$. Such sums are also called (noncommutative) \DEF{polynomials}.
It becomes an $\R$-algebra when taking the \DEF{concatenation} product
of words (and extending it linearly).
We shall also use the space of formal tensor series, $T(( \R^ d ))$,
which contains \emph{all} formal sums of words  of form \eqref{eq:formal} with possibly infinite number of  coefficients $c_w \in \R$.

Define the two-sided ideal
\begin{align*}
  T_{>L}(\R^d)
  := 
  \left\{ \sum_{w \in A^*} c_w\ w \in T(\R^d) \mid \forall w\in A^*: |w|\le L\Rightarrow c_w = 0 \right\}.
\end{align*}

The \DEF{truncated tensor algebra} is the quotient algebra
\begin{align*}
  T_{\le L}( \R^ d ) := \faktor{T(\R^d)}{T_{>L}(\R^d).}
\end{align*}
It can be realized as the space of formal sums \eqref{eq:formal}
satisfying $c_w = 0$ for $|w| > L$.
The product on it is the concatenation product,
where the product of two words $w,v$ with $|w| + |v| > L$
is set to zero.
We will use this identification from now on.

Like every associative algebra,
the tensor algebra, as well as its truncation, is
endowed with a Lie bracket given by the commutator
  $[v,w] = vw - wv$.

The \DEF{free Lie algebra} (over $\R^d$) ,
$\mathfrak{g}(\R^d)$,
can be realized
as the smallest sub-Lie algebra of $T(\R^d)$ containing
the letters $A$.

The \DEF{free, step-$L$ nilpotent Lie algebra} (over $\R^d$),
$\mathfrak{g}_{\le L}(\R^d)$,
can be realized
as the smallest sub-Lie algebra of $T_{\le L}(\R^d)$ containing
the letters $A$.

The \DEF{free, step-$L$ nilpotent Lie group} (over $\R^d$),
$\mathcal{G}_{\le L}(\R^d)$,
can be realized as the image of
$\mathfrak{g}_{\le L}(\R^d)$ under the exponential map
\begin{align*}
  \exp: T_{\le L}(\R^d) &\to T_{\le L}(\R^d) \\
          z             &\mapsto \sum_{k=0}^\infty \frac{z^k}{k!}.
\end{align*}
Its product is given by the restriction of the concatenation.
As the name suggests, $\mathcal{G}_{\le L}(\R^d)$ is a Lie group.
Moreover, its Lie algebra is realized by 
$\mathfrak{g}_{\le L}(\R^d)$ and the map $\exp$
realizes the Lie group exponential.
Moreover $\exp: \mathfrak{g}_{\le L}(\R^d) \to \G_{\le L}(\R^d)$
is a \emph{global} diffeomorphism, with inverse given by
\begin{align*}
  \log: \G_{\le L}(\R^d) &\to \mathfrak{g}_{\le L}(\R^d) \\
    \mathbf{g} &\mapsto \sum_{k=1}^\infty (-1)^{k+1} \frac{ (\mathbf{g} - \e)^k }k.
\end{align*}

Let $X: [0,\TIME] \to \R^d$ be given.
The iterated-integrals signature, truncated at level $L$, is an element
\begin{align*}
  \IIS_{\le L}( X ) =
  \sum_w c_w\ w
  \in
  T_{\le L}(\R^d),
\end{align*}
with coefficients given, for a word $w=w_1 \dots w_k$, as
\begin{align*}
  c_w 
  &= \int_{0 \le t_1  \le \cdots \le t_k \leq \TIME}\mathrm dX_{t_1}^{(w_1)} \dots \mathrm dX_{t_k}^{(w_k)} \\
  &= \int_{0 \le t_1  \le \cdots \le t_k \leq \TIME} \dot X_{t_1}^{(w_1)} \dots \dot X_{t_k}^{(w_k)}\,\mathrm dt_1 \dots \mathrm dt_k.\notag
\end{align*}

\begin{theorem}[{\cite[Theorem 7.30]{bib:FV2010}}]
  $\IIS_{\le L}( X )$ is an element of $\G_{\le L}( \R^d )$.
\end{theorem}

\subsection{Baker-Campbell-Hausdorff formula}
We will use the following classical result
(see \cite[Theorem 7.24]{bib:FV2010} for a proof in a notation close to ours).
\begin{theorem}[Baker-Campbell-Hausdorff formula]
  \label{thm:bch}

  Let $X, Y$ be non-commuting
  dummy variables (i.e. we work in the free Lie algebra over two letters $X,Y$).
  Then
  \begin{align*}
    \log( \exp(X) \exp(Y) ) = X + \int_0^1 \Theta\left( \exp( t \ad_X ) \exp( \ad_Y ) \right) X \mathrm dt.
  \end{align*}
  Here, $\ad_X$ is the linear map,
      $\ad_X( Z ) = [X, Z]$,
  and
  \begin{align*}
    \Theta(z) := \sum_{n \ge 0} \frac{(-1)^n}{n+1} (z-1)^n.
  \end{align*}

  Spelling out the first few terms of the resulting Lie series:
  \begin{align*}
    \log( \exp(X) \exp(Y) )
    &=
    X + Y  + \frac{1}{2}[X,Y] + \frac{1}{12}[X,[X,Y]] - \frac{1}{12}[Y,[X,Y]] \\
    &\quad - \frac {1}{24}[Y,[X,[X,Y]]]
           - \frac{1}{720}\left([[[[X,Y],Y],Y],Y] +[[[[Y,X],X],X],X]\right) \\
     &\quad + \frac{1}{360}([[[[X,Y],Y],Y],X]+[[[[Y,X],X],X],Y]) \\
     &\quad + \frac{1}{120}([[[[Y,X],Y],X],Y] +[[[[X,Y],X],Y],X]) + \cdots. 
  \end{align*}

\end{theorem}
The  BCH formula for permuted arguments is simply
\begin{equation}\label{eq:bch_permutation}
\BCH(Y,X) = \log(\exp(Y)\exp(X)) = - \log((\exp(Y)\exp(X))^{-1}) = -\BCH(-X,-Y).
\end{equation}
For multiple arguments, the BCH formula can be obtained by composition, which is associative via
\begin{equation}\label{eq:bch_associativity}
\log(\exp(\BCH(X,Y))\exp(Z))=\log(\exp(X)\exp(Y)\exp(Z))=\BCH(X,\BCH(Y,Z)).
\end{equation}

In the following subsection, we recall the Lyndon basis and its dual, as they are useful to compute $\BCH(X,Y)$.




\subsection{Basis of the truncated Lie algebra and its dual}\label{sec:sigAndLogSig}
A word 
$w$ over $d$ symbols $\{\w{1},\ldots,\w{d}\}$ is a \DEF{Lyndon word} if and only if it is nonempty and lexicographically strictly smaller than any of its proper suffixes, that is 
   $ w
<
v$
for all nonempty words 
$v$
 such that 
$w
=
u
v$
and 
$u$ is nonempty.
For all Lyndon words
of length larger equal to $2$ (i.e. non-letter Lyndon words),
there is a unique choice of
$u$
and $v$, called the \DEF{standard factorization} of $w$, in which 
$v$ is as long as possible and both $u$ and $v$ are Lyndon words. 
For every Lyndon word $w$ we obtain a polynomial $\mathcal{B}_w\in T(\R^d)$ via  the recursive definition
$$\mathcal{B}_w=\begin{cases}\w{j}&\text{if }w=\w{j}\\
[\mathcal{B}_u,\mathcal{B}_v]&\text{if $w$ has standard factorization }w=uv.\end{cases}$$
The Lyndon words, sorted by length first and then lexicographically within each length class, form an infinite sequence $(w_i)_{i\in\N}$. 
For convenience we set $\mathcal{B}_i:={\mathcal{B}}_{w_i}$ for every $i\in\N$.  
Let $\LieAlg=\LieAlg(\w{1},\ldots,\w{d})=\mathfrak{g}(\R^d)$ be the free Lie algebra over $\R$ which is generated by  $\{\w{1},\ldots,\w{d}\}$. 
Then $(\mathcal{B}_i)_{i\in\N}$ forms an $\R$-basis for $\LieAlg$
(\cite[Section 3,4]{bib:Gar1990}). 

For the entire section, we fix a truncation level $L$. 
The $L$-truncated Lie algebra $\LieAlg_{\leq L}$ with nested Lie brackets bounded by depth $L$ is an $\R$-vector space (\cite[Proposition 3.1]{bib:Gar1990}) with 
\begin{equation}\label{eq:dim}
B:=B_{L,d}:=
\dim_\R(\LieAlg_{\leq L})=\sum_{1\leq \ell\leq L}\frac1\ell\sum_{a\mid \ell}\mu(a)d^{\frac{\ell}{a}},\end{equation}
where $\mu$ denotes the M\"obius function. 
An $\R$-basis is given by the $L$-truncated Lyndon basis $\mathcal{B}_{1\leq b\leq B}$ which is recalled above.

\begin{example}\label{ex:ComputeB}
For $L=3$ and $d=2$ we obtain $B=5$ and
$$\mathcal{B}_{1\leq b\leq B}=\begin{bmatrix}
\w{1}&
\w{2}&
\left[\w{1},\w{2}\right]&
\left[\w{1},\left[\w{1},\w{2}\right]\right]&
\left[\left[\w{1},\w{2}\right],\w{2}\right]\end{bmatrix}\in\LieAlg_{\leq L}^5.$$
\end{example}

On the corresponding $B$-dimensional dual space 
$$\hom(\LieAlg_{\leq L},\R)=\{f:\LieAlg_{\leq L}\rightarrow \R\mid f\text{ $\R$-linear}\}$$
we use the dual basis
 $\mathcal{B}_{1\leq b\leq B}^*$, satisfying  for all $b,j\leq B$, 
\begin{equation}\label{eq:dualBasis}
\mathcal{B}^*_j(\mathcal{B}_b)=\begin{cases}1&\text{ if }j=b\\0&\text{ elsewhere.}\end{cases}
\end{equation}

\subsection{BCH in the truncated Lie algebra}
With some abuse of notation, we define for $X, Y\in\LieAlg_{\leq L}$ the $L$-truncated  Lie series 
$$\BCH(X,Y):=\log( \exp(X) \exp(Y) )\in\LieAlg_{\leq L}$$
derived from the Baker-Campbell-Hausdorff formula (\Cref{thm:bch}).

\begin{definition}\label{def:groupLaw}
We define a binary operation $\star:\R^B\times\R^B\rightarrow \R^B$ of coefficient vectors via  
\begin{align*}
{\left(u\star v\right)}_i
:=\mathcal{B}^*_i\circ\BCH\left(\sum_{j=1}^Bu_j\mathcal{B}_j,\sum_{j=1}^Bv_j\mathcal{B}_{j}\right)
\end{align*}
for every $i\leq B$ and $u,v\in \R^B$. 
\end{definition}

\begin{example}\label{Ex:heisenberggroup}
    For $L=d=2$ we have $B=B_{2,2}=3$ and 
    \begin{align*}
&\begin{bmatrix}u_1&u_2&u_{3}\end{bmatrix}
\star\begin{bmatrix}v_1&v_2&v_{3}\end{bmatrix}
=\begin{bmatrix}u_1+v_1&u_2+v_2&u_3+v_{3}+\frac{1}{2}(u_1v_2-u_2v_1)\end{bmatrix}
\end{align*}
according to \Cref{def:groupLaw}. 
\end{example}

\begin{lemma}\label{lem:grouplaw}
$\left(\R^B,\star,0_B\right)$ is a group. 
\end{lemma}
\begin{proof}
 It is easy to see that $0_B\in\R^B$ is the neutral element with respect to $\star$ thanks to the property
\[
\BCH(X,0)=X=\BCH(0,X).
\]
Next,   for every $u\in\R^B$, its inverse with respect to $\star$ is nothing but $-u$  thanks to
$$\BCH(X,-X)=0=\BCH(-X,X).$$ 
Finally, associativity with respect to $\star$ follows from \eqref{eq:bch_associativity}. 
\end{proof}

\begin{lemma}
    In the setting of \Cref{Ex:heisenberggroup}, that is $L=d=2$ and $B=3$, the group $\left(\R^3,\star,0_3\right)$ according to \Cref{lem:grouplaw} is isomorphic to the Heisenberg group
    $$H:=\left\{\begin{bmatrix}
        1&a_1&a_3\\
        0&1&a_2\\
        0&0&1
    \end{bmatrix}\in\GL_3(\R)\mid a_1,a_2,a_3\in\R\right\}$$
    with its multiplication inherited from $\GL_3(\R)$. 
\end{lemma}

\begin{proof}
    Define $\Phi:\R^3\rightarrow H$ via  
    $$\Phi\left(\begin{bmatrix}
        u_1&u_2&u_3
    \end{bmatrix}\right):=\begin{bmatrix}
             1&u_1&u_3+\frac{1}{2}u_1u_2\\
        0&1&u_2\\
        0&0&1   
    \end{bmatrix}$$
    for every $u_1,u_2,u_3\in\R$. 
    It is bijective, and it respects the group law with 
    \begin{align*}
    \Phi(\begin{bmatrix}
        u_1&u_2&u_3
    \end{bmatrix})\,\Phi(\begin{bmatrix}
        v_1&v_2&v_3
    \end{bmatrix})
    &=
    \begin{bmatrix}
             1&u_1+v_1&u_3+v_3+u_1v_2+\frac{1}{2}(u_1u_2+v_1v_2)\\
        0&1&u_2+v_2\\
        0&0&1   
    \end{bmatrix} \\
    &=
    \Phi\left(\begin{bmatrix}u_1&u_2&u_{3}\end{bmatrix}
\star\begin{bmatrix}v_1&v_2&v_{3}\end{bmatrix}\right)
  \end{align*}
  for all $v_1,v_2,v_3\in\R$. Therefore,  $\Phi$ is an isomorphism of groups. 
\end{proof}


\section{The barycenter in the nilpotent Lie group}\label{sec:barycenterSymbolic}
\subsection{Definition and properties}

\begin{definition}[{\cite[Definition 8.1.4]{bib:BK1981},\cite[Definition 11]{bib:PL2020}}] \label{def:groupmean}
  Let $G$ be a Lie group with globally defined logarithm%
  \footnote{In the case that we are interested in, this condition
  is satisfied, and we can thus omit the usual assumption that $\nu$ is supported
  in a neighborhood of the identity, where the logarithm is well-defined.}
  and $\nu$ a probability measure on it.
  We say that $\mathbf{m}_\nu \in G$ is a \textbf{barycenter} or \textbf{group mean}
  of $\nu$
  if
  \begin{align}
  \label{eq:barycenterCondition}
    0 = \int_G \log( \mathbf{m}_\nu^{-1} \mathbf{x} ) \,\nu(\mathrm d\mathbf{x}).
  \end{align}
\end{definition}
The notion of barycenter was introduced in \cite{bib:PL2020} using the Cartan-Schouten connection.
Informally speaking, the barycenter looks for a point $\mathbf{m}_\nu$ so that  the logarithm%
\footnote{The left-hand side is the abstract logarithm at a basepoint in a Lie group; the right-hand side is its concrete realization inside the tensor algebra.}
$$\log_{\mathbf{m}_\nu}\mathbf{x} = \mathbf{m}_\nu\log( \mathbf{m}_\nu^{-1} \mathbf{x} )$$ 
with respect to $\mathbf{m}_\nu$ has expectation $0$.
This notion is different from the so-called naive mean, which simply averages the logarithms of the points at the identity,
namely
\begin{align*}
    \mathbf{m}^{\operatorname{naive}}_\nu
    :=
    \exp\left( \,\int_G \log(\mathbf{x}) \,\nu(\mathrm d\mathbf{x}) \right).
\end{align*}
Note that, in general, $\mathbf{m}^{\operatorname{naive}}_\nu \neq \mathbf{m}_\nu$ and $\mathbf{m}^{\operatorname{naive}}_\nu$ does not possess invariance properties.

\begin{remark}~
\label{rem:groupmean}
\begin{enumerate}
\item 
For compact Lie groups, the notion of barycenter corresponds to the Riemannian center of mass for the corresponding bi-invariant Riemannian metric.
However, in our case, it is impossible to define the bi-invariant Riemannian metric \cite{bib:PL2020}.

\item
The barycenter of \Cref{def:groupmean} formally fits into the framework
of proper scoring rules and Bayes acts
\cite{bib:good1952rational,bib:brehmer2020properization,bib:BO2021}.
In that setting, a Bayes act is defined as
\begin{align*}
    a_\mu := \argmin_a \E_{X\sim \mu}\left[ L(a, X) \right],
\end{align*}
for some loss function $L$.
If $G$ is a Riemannian manifold, if the Riemannian
logarithm coincides with the Lie group logarithm,
and if we set
\begin{align*}
    L(a,X) := ||\log_a( \IIS(X))||^2,
\end{align*}
then the condition for the minimum is
\begin{align*}
    0 = \E_{X\sim \mu}[ \partial_a ||\log_a( \IIS(X))||^2 ]
    = - 2 \E_{X\sim \mu}[ \log_a( \IIS(X) ) ],
\end{align*}
which, modulo the irrelevant prefactor, is exactly condition \eqref{eq:barycenterCondition}.

Now, as just mentioned, our $G$  of interest is \emph{not} Riemannian (it can only be endowed with
a compatible sub-Riemannian geometry) and therefore
this formal argument does \emph{not} apply.
It would be nonetheless interesting to explore what ideas,
in particular, the concept of  ``elicitation'',
from that literature can be applied in our setting.
\end{enumerate}
\end{remark}

Our main results show that for the free Lie groups, the barycenter exists and is unique under standard conditions.
\begin{theorem}
\label{thm:main}
Let $G = \G_{\le L}(\R^d)$ be the free, step-$L$ nilpotent Lie group (over $\R^d$).
Let $\nu$ be a probability measure on $G$ such that this measure is integrable when
considered as a measure on the ambient linear space $T_{\le L}(\R^d)$.
Then the group mean $\mathbf{m}_\nu$ of $\nu$ exists and is unique.
\end{theorem}
The proof of \Cref{thm:main} is postponed to \Cref{sec:mainProof}.

\begin{remark}
For compactly supported measures, the uniqueness of the barycenter follows from  \cite[Example 8.1.8]{bib:BK1981} (see also \cite[Theorem 5.16]{bib:PL2020}).
For this, note  that the free nilpotent Lie group of dimension $d$ of step $L$ is simply connected since it is  diffeomorphic to the free Lie algebra  of dimension $d$ of step $L$, \cite{bib:FV2010}.

However, \Cref{thm:main} is stronger in the sense that it provides a constructive way to compute  the barycenter in a recursive fashion (and also
   covers measures that are \emph{not necessarily compactly supported}).
\end{remark}

\begin{remark}
   If $\mu$ is a measure on $\mathcal G(\R^d)$ (with all moments well defined),
   such that the pushforward $\mu_{\le L}$ under the projection 
   onto levels $\le L$ has
   the appropriate integrability conditions for all $L$,
   then the proof shows that level $n$ of the barycenter 
   is independent of $L$ for $L \ge n$.
   We thus get, projectively, a well-defined barycenter
   for all of $\mu$.
\end{remark}

\begin{theorem}[Bi-invariance of the barycenter]\label{thm:invariance}
Under assumptions of \Cref{thm:main},
     for $\mathbf{g} \in G$ denote by $(L_{\mathbf{g}})_* \nu$ the push-forward
      under left multiplication, i.e., for any bounded function $f$,
      \begin{align*}
        \int_G f( \mathbf{y} ) \,((L_{\mathbf{g}})_*\nu)(\mathrm d\mathbf{y}) = \int_G f(\mathbf{g} \mathbf{x}) \,\nu(\mathrm d\mathbf{x}).
      \end{align*}
      Let $(R_{\mathbf{g}})_* \nu$ be the push-forward under right multiplication.
      Then $(L_{\mathbf{g}})_* \nu, (R_{\mathbf{g}})_* \nu$ are also integrable in the required sense
      and
      \begin{align*}
        \mathbf{m}_{(L_{\mathbf{g}})_* \nu} = \mathbf{g} \mathbf{m}_{\nu},
        \qquad
        \mathbf{m}_{(R_{\mathbf{g}})_* \nu} = \mathbf{m}_{\nu} \mathbf{g}.
      \end{align*}
 That is, the barycenter is  \textbf{bi-invariant} with respect to left and right multiplication, illustrated by \Cref{fig:biinvariant}.
\end{theorem}


\begin{figure}
    \centering
    \includegraphics[width=0.7\linewidth]{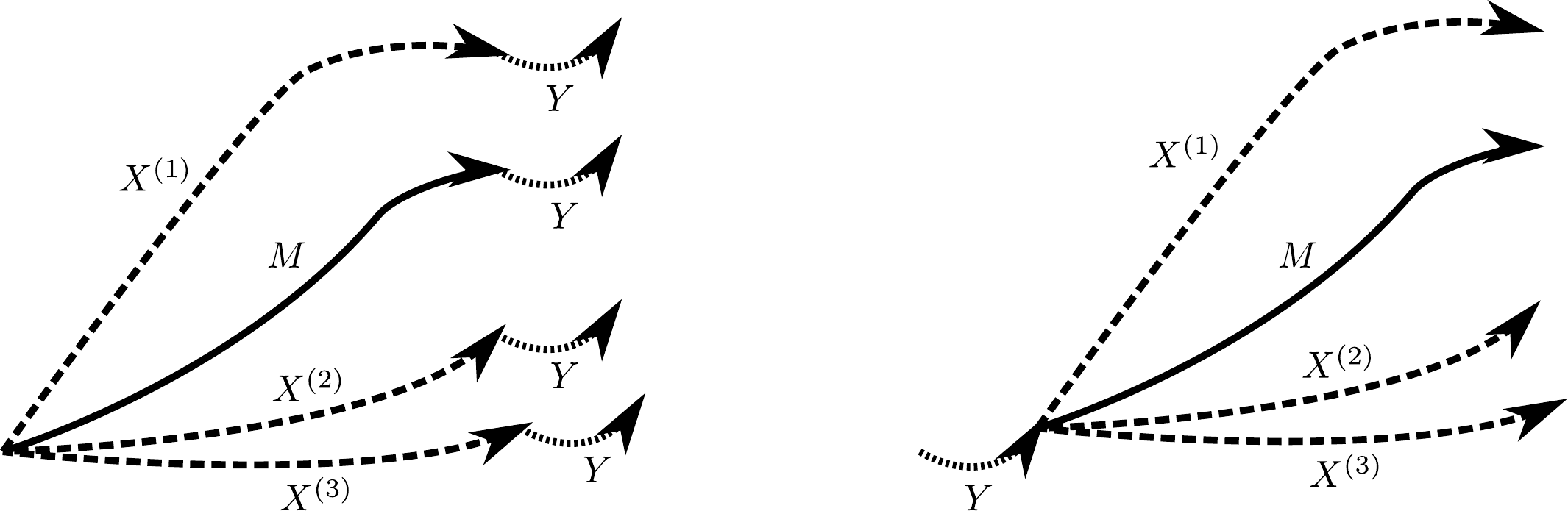}
    \caption{On the left the mean of the signatures $\IIS(X^{(i)})$ of three
    paths is assumed to be given by the signature $\IIS(M)$ of some path $M$.
   If we attach to all paths $X^{(i)}$ a new path segment
   $Y$, since the mean is right invariant,
   this corresponds to attaching that path segment to $M$.
   On the right, we see the analogous visualization for
   left invariance. Due to Chen's identity, attaching path corresponds to the group product.}
    \label{fig:biinvariant}
\end{figure}

The proof of \Cref{thm:invariance} mainly follows from \cite[Theorem 5.13]{bib:PL2020}, but we provide it below for completeness.

\begin{proof}[Proof of \Cref{thm:invariance}]
  First,
  \begin{align*}
    \int \log( (\mathbf{g} \mathbf{m}_\nu)^{-1} \mathbf{x} ) \,(L_{\mathbf{g}})_* \nu(\mathrm d\mathbf{x})
    =
    \int \log( (\mathbf{g} \mathbf{m}_\nu)^{-1} \mathbf{g} \mathbf{x} ) \,\nu(\mathrm d\mathbf{x})
    =
    \int \log( \mathbf{m}_\nu^{-1}\mathbf{x} ) \,\nu(\mathrm d\mathbf{x})
    =
    0.
  \end{align*}
  Hence $\mathbf{m}_{(L_{\mathbf{g}})_* \nu} = \mathbf{g} \mathbf{m}_{\nu}$.
  Further
  \begin{align*}
    \int \log( (\mathbf{m}_\nu \mathbf{g})^{-1} \mathbf{x} ) \,(R_{\mathbf{g}})_* \nu(\mathrm d\mathbf{x})
    &=
    \int \log( (\mathbf{m}_\nu \mathbf{g})^{-1} \mathbf{x g} ) \,\nu(\mathrm d\mathbf{x})
    =
    \int \log( {\mathbf{g}}^{-1} \mathbf{m}_\nu^{-1}\mathbf{x g} ) \,\nu(\mathrm d\mathbf{x}) \\
    &=
    {\mathbf{g}}^{-1} \left(\;\int \log( \mathbf{m}_\nu^{-1}\mathbf{x} ) \,\nu(\mathrm d\mathbf{x})\right)\mathbf{g}
    =
    0.
  \end{align*}
  Here we used
  \begin{align*}
    \log({\mathbf{g}}^{-1} \mathbf{x g}) = {\mathbf{g}}^{-1} \log(\mathbf{x}) \mathbf{g},
  \end{align*}
  which can, for example, be verified
  by expanding the power series for $\log$.
\end{proof}

Finally, we show that in the special case of a ``centered'' probability measure,
the barycenter coincides with the naive mean.
\begin{lemma}\label{lem:barycenter_naive}
  In the setting of \Cref{thm:main} we have
  \begin{align*}
    \mathbf{m}_\nu = \e  \iff  \mathbf{m}^{\operatorname{naive}}_\nu = \e.
  \end{align*}
\end{lemma}
\begin{proof}
  Assume $\mathbf{m}^{\operatorname{naive}}_\nu = \e$.
  Then
  \begin{align*}
    0
    = \log( \e )
    = \log( \mathbf{m}^{\operatorname{naive}}_\nu )
    = \int_G \log(\mathbf{x}) \,\nu(\mathrm  d\mathbf{x})
    = \int_G \log(\e^{-1} \mathbf{x}) \,\nu(\mathrm d\mathbf{x}).
  \end{align*}
  Hence, by uniqueness of $\mathbf{m}_\nu$, $\mathbf{m}_\nu = \e$.

  Assume $\mathbf{m}_\nu = \e$.
  Then
  \begin{align*}
    0 = \int \log(\mathbf{x}) \,\nu(\mathrm d\mathbf{x}),
  \end{align*}
  and hence $\mathbf{m}^{\operatorname{naive}}_\nu = \exp(0) = \e$.
\end{proof}

\subsection{Key lemma}
The key idea is that using the BCH formula implies the following polynomial relations on the coefficients in the expansion.
Let $x,y \in \mathfrak{g}_{\leq L}(\R^d)$ be as follows,
\[
x = \sum_{j=1}^B u_j\mathcal{B}_j, \quad y = \sum_{j=1}^B v_j\mathcal{B}_j,
\]
where the order of the chosen basis $\mathcal{B}$ of $\mathfrak{g}_{\le L}(\R^d)$ respects the length.
Then, by plugging $x,y$ as $X,Y$ in the BCH formula, we get
  \begin{align*}
\exp(x) \exp(y) = \exp\left(\sum_{j=1}^B
      \left( u_j + v_j + p_j(u_1, \dots, u_{j-1}, v_1,\dots,v_{j-1}) \right)
    \mathcal{B}_j \right),
  \end{align*}
where $p_j$ are some polynomials that are globally defined for fixed $L$, $d$, and the chosen basis $\mathcal{B}$  (see the formal construction below).

Formally, let 
$$R:=\R[M_1,\dots,M_B,C_1,\dots,C_B]$$
be the polynomial algebra over $2B$ symbols and  $\LieAlgSymb
:= \LieAlgSymb(\w{1},\ldots,\w{d})
:= \mathfrak{g}(R^d)$ be the free Lie algebra with coefficients taken from $R$. 
Its truncation $\LieAlgSymb_{\leq L}$ is a $B$-dimensional free $R$-module with 
$L$-truncated Lyndon basis $\mathcal{B}$. 
 On the corresponding dual space 
$$\hom(\LieAlgSymb_{\leq L},R)=\{f:\LieAlgSymb_{\leq L}\rightarrow R\mid f\text{ $R$-linear}\}$$
we use the dual basis
 $\mathcal{B}^*$ as in  (\ref{eq:dualBasis}).
We now apply the Baker--Campbell--Hausdorff formula as in the group law of  \Cref{def:groupLaw}, but for elements from $\LieAlgSymb_{\leq L}$.

\begin{lemma}\label{lem:pol_ex}
Consider the following elements of $\LieAlgSymb_{\leq L}$,
\begin{align*}
  X := \sum_{b=1}^BM_b{\mathcal{B}}_b, \quad
  Y := \sum_{b=1}^BC_b{\mathcal{B}}_b.
\end{align*}
Then, for every $j\leq B$, there exist polynomials 
\begin{align}\label{symbolic_bch}
{\mathcal{B}}^*_j\circ\BCH(X,Y)
=M_j+C_j+p_j\end{align}
for (uniquely determined) polynomials $p_j\in\R[M_1,\dots,M_{j-1},C_1,\dots,C_{j-1}]$. 
\end{lemma}

\begin{proof}
  The Lyndon words are sorted by length, and each Lie bracket strictly increases the depth of nested Lie brackets. Therefore, together with \Cref{thm:bch}, we have 
  \begin{align}
  {\mathcal{B}}^*_j\circ\BCH(X,Y)
  =M_j+C_j+{\mathcal{B}}^*_j\left(
  \BCH(X_{<j},Y_{<j})-(M_j+C_j)\mathcal{B}_j\right)\end{align}
  for all $j\leq B$, where
  \begin{align*}
    X_{<j} := \sum_{b=1}^{j-1}M_b{\mathcal{B}}_b, \quad
    Y_{<j} := \sum_{b=1}^{j-1}C_b{\mathcal{B}}_b.
  \end{align*}
  The claim follows with $p_j:={\mathcal{B}}^*_j\left(\BCH(X_{<j},Y_{<j})-(M_j+C_j)\mathcal{B}_j\right)$.
\end{proof}
\begin{example}\label{ex:polys_by_pch}
In the setting of \Cref{ex:ComputeB} with $L=3$ and $d=2$, 
  $$
\begin{bmatrix}p_1\\p_2\\p_3\\p_4\\p_5\end{bmatrix}=
\begin{bmatrix}0\\
0\\
\frac{1}{2}C_{2}M_{1} - \frac{1}{2}C_{1}M_{2}\\
-\frac{1}{12}C_{1}C_{2}M_{1} + \frac{1}{12}C_{2}M_{1}^2 + \frac{1}{12}C_{1}^2M_{2} - \frac{1}{12}C_{1}M_{1}M_{2} + \frac{1}{2}C_{3}M_{1} - \frac{1}{2}C_{1}M_{3}\\
\frac{1}{12}C_{2}^2M_{1} - \frac{1}{12}C_{1}C_{2}M_{2} - \frac{1}{12}C_{2}M_{1}M_{2} + \frac{1}{12}C_{1}M_{2}^2 - \frac{1}{2}C_{3}M_{2} + \frac{1}{2}C_{2}M_{3}
\end{bmatrix}\in R^5.
$$
\end{example}

\subsection{Proof of the main theorem}\label{sec:mainProof}
\begin{proof}[Proof of \Cref{thm:main}]
Let $\mathcal{B}$ denote the $L$-truncated Lyndon basis. 
Assume there is a group element
$$\mathbf{m}=\exp\left(\sum_{j=1}^Bm_j\mathcal{B}_j\right)$$ 
which satisfies 
\begin{align*}
   0
   &= \int_G\log(\mathbf{m}^{-1} \mathbf{x})\,\nu(\mathrm d\mathbf{x})\\
   &= \int_G\log\left(\exp\left(-\sum_{j=1}^Bm_j\mathcal{B}_j\right)\,\exp\left(\sum_{j=1}^Bc^{(\mathbf{x})}_j\mathcal{B}_j\right)\right)\nu(\mathrm d\mathbf{x})
\end{align*}
with suitable coefficients $c^{(\mathbf{x})}_j$ for $\mathbf{x}\in G$ and $1\leq j\leq B$.
Then via \Cref{thm:bch}, 
\begin{align*}
   0
    &= \int_G\log\circ\exp\left(\sum_{j=1}^B\left(p_j(-m_1,\ldots,-m_{j-1},c^{(\mathbf{x})}_1,\ldots,c^{(\mathbf{x})}_{j-1})-m_j+c^{(\mathbf{x})}_j\right)\mathcal{B}_j\right)\,\nu(\mathrm d\mathbf{x})\\
    &= \sum_{j=1}^B\left(\int_G \left(p_j(-m_1,\ldots,-m_{j-1},c^{(\mathbf{x})}_1,\ldots,c^{(\mathbf{x})}_{j-1})-m_j+c^{(\mathbf{x})}_j\right)\,\nu(\mathrm d\mathbf{x})\right)\mathcal{B}_j
\end{align*}
with polynomials $p_j$ according to \Cref{lem:pol_ex}.
Since $\mathcal{B}$ is an $\R$-basis, we obtain
\begin{align}
  \label{eq:mi}
  m_j =
  \int_G p_j(-m_1,\ldots,-m_{j-1},c^{(\mathbf{x})}_1,\ldots,c^{(\mathbf{x})}_{j-1})+c^{(\mathbf{x})}_j \,\nu(\mathrm d\mathbf{x}),
\end{align} 
and thus iteratively
the components of $m$.\footnote{This integral is well-defined,
by the integrability assumption.
The easiest way to see this is via
an expression we obtain for it
that is linear in $\mathbf{x}$, see \Cref{lem:integral_existence} in \Cref{sec:pi1}.}
Hence, $\mathbf{m}$ is unique.

It is immediate to see that \emph{defining} $\mathbf m$ via \eqref{eq:mi} also yields existence.
\end{proof}

%




A very common case  is when the measure $\nu$ is discrete, i.e., it is supported on $N$ points with weights $w^{(i)}$, where
\begin{equation}\label{eq:discrete_measure}
\nu(\{\mathbf{x}^{(1)}\}) =w^{(1)}, \ldots,\nu(\{\mathbf{x}^{(N)}\})=w^{(N)},\quad \sum\limits_{i=1}^{N} w^{(i)} = 1.
\end{equation}
In this case, the following corollary provides an algorithm for iterative computation.

\begin{corollary}\label{th:groupmean}
For the discrete measure \eqref{eq:discrete_measure}, where the points have expansions
$$\mathbf{x}^{(i)}=\exp\left(\sum_{j=1}^{B}c_{j}^{(i)}\mathcal{B}_j\right), \quad c^{(i)}\in\R^{B},\quad 1\leq i\leq N,$$
the coefficient vector $m$ of the mean according to~\Cref{def:groupmean} 
can be computed recursively using the closed formula 
\begin{equation}\label{eq:mj}
    m_j=\sum_{i=1}^N w^{(i)} \left(c_{j}^{(i)}+p_j(-m_1,\ldots,-m_{j-1},c_{0}^{(i)},\ldots,c_{j-1}^{(i)})\right)
\end{equation}
for every $j\leq B$, see \Cref{algo:groupmean-grouplike}.
\end{corollary}
\begin{proof}
Due to  \Cref{lem:grouplaw},  
   \begin{align*}
   0 &= \sum_{i=1}^N w^{(i)}\log( \mathbf{m}^{-1} \mathbf{x}^{(i)} )=\sum_{i=1}^Nw^{(i)}\sum_{j=1}^B{((-m)\star c^{(i)})}_{j}\mathcal{B}_j\\
&=\sum_{j=1}^B \sum_{i=1}^N w^{(i)} 
\left(-m_{j}+c_{j}^{(i)}+p_j(-m_{1},\ldots,-m_{j-1},c_{1}^{(i)},\ldots,c_{j-1}^{(i)})\right)\mathcal{B}_j
\end{align*}
with polynomials $p_j\in\R[M_1,\ldots,M_{j-1},C_1,\ldots,C_{j-1}]$ according to \Cref{lem:pol_ex}, and  where $p_1=\cdots=p_d=0$. 
 
\end{proof}

\LinesNotNumbered
\begin{algorithm2e}\label{algo:groupmean-grouplike}
\DontPrintSemicolon 
\KwIn{A set of \(N\) coefficient vectors $x^{(i)}\in\R^B$ for group elements ${(\mathbf{x}^{(i)})}_{1\leq i\leq N}
$}
\KwOut{The coefficients $m\in\R^B$ for the group mean $\mathbf{m}$}
\setcounter{AlgoLine}{0}
\nl Precompute polynomials \({(p_j)}_{1\leq j\leq B}\)\;
\nl \For{\(j=1,\dots,B\)}{
    \nl Compute \(m_j\) using \Cref{eq:mj}\; 
}
\nl \Return{\(m=\begin{bmatrix}m_1&\dots&m_B\end{bmatrix}\)}\;
\caption{\textsc{Group mean}}
\end{algorithm2e}

\begin{example}\label{ex:closedFormualsForL3d2}
  Continuing \Cref{ex:polys_by_pch}, and assuming $w^{(i)} = \frac1N$ for simplicity, 
  \begin{small}
  \begin{align}
    m_1 &= \frac1N \sum_{i=1}^Nc_{1}^{(i)},
    \quad m_2 = \frac1N \sum_{i=1}^Nc_{2}^{(i)},\quad 
    m_3 = \frac1N \sum_{i=1}^N \left(-\frac{1}{2}c_{2}^{(i)}m_{1} + \frac{1}{2}c_{1}^{(i)}m_{2} + c_{3}^{(i)} \right)\nonumber\\
       m_4 &=\frac1N \sum_{i=1}^N\left(\frac{1}{12}c_{1}^{(i)}c_{2}^{(i)}m_{1} + \frac{1}{12}c_{2}^{(i)}m_{1}^2 - \frac{1}{12}c_{1}^{(i)}c_{1}^{(i)}m_{2} - \frac{1}{12}c_{1}^{(i)}m_{1}m_{2} - \frac{1}{2}c_{3}^{(i)}m_{1} + \frac{1}{2}c_{1}^{(i)}m_{3} + c_{4}^{(i)} \right),\nonumber\\
       m_5 
       &=\frac1N \sum_{i=1}^N\left(-\frac{1}{12}c_{2}^{(i)}c_{2}^{(i)}m_{1} + \frac{1}{12}c_{1}^{(i)}c_{2}^{(i)}m_{2} - \frac{1}{12}c_{2}^{(i)}m_{1}m_{2} + \frac{1}{12}c_{1}^{(i)}m_{2}^2 + \frac{1}{2}c_{3}^{(i)}m_{2} - \frac{1}{2}c_{2}^{(i)}m_{3} + c_{5}^{(i)}\right),\nonumber
  \end{align}
  \end{small}
which gives all the steps of \Cref{algo:groupmean-grouplike}.
\end{example}

Given a truncation level \(L\) and the dimension \(d\) of the inputs, we obtain $B=B_{L,d}$ according to \Cref{eq:dim}.
Then the computational complexity of  \Cref{algo:groupmean-grouplike} is given by the following  lemma.

\begin{lemma}\label{thm:complexity}
There is a $Q_L$ such that for all $d$ and  
$\mathbf{x}^{(1)},\ldots,\mathbf{x}^{(N)}\in \mathcal{G}_{\leq L}(\R^d)$  we can compute the  empirical group mean 
according to \Cref{def:groupmean} in less then 
$$NB_{L,d}\left(1+LQ_{L}\right)$$
basic operations. 
The memory complexity is trivial, i.e., we only have to store the 
$B_{L,d}$ coefficients $m_j$ during runtime.
\end{lemma}
\begin{proof}
For every $1
\leq j\leq B$ let $Q_{L,d,j}$ denote the number of sumands in $p_j$.
    Let us for the moment assume $d=L$ and 
    set $Q_L=\max_{j}Q_{L,L,j}$.  
    Now assume $d>L$. 
    Every Lyndon word $w=v_1\dots v_{\ell}$, with $v_t\in\{\w{1},\dots,\w{d}\}$ and $\ell\leq L$, can be written as $\varphi (v_1)\dots \varphi (v_\ell)$ with $\varphi( v_t)\in\{\w{1},\dots,\w{L}\}$, and where $\varphi$ is an homomorphism of monoids which preserves the order of letters. 
    Clearly $w$ and $\varphi(w)$ lead to polynomials with the same number of summands, bounded by $Q_L$.
Hence, we can evaluate \Cref{eq:mj} in $$N\sum_{1\leq j\leq B}1+\deg(p_j)Q_{L,d,j}\leq NB(1+LQ_{L})$$
basic operations, where $\deg(p_j)$ is bounded by the truncation level $L$.
\end{proof}

\begin{remark}
The family of polynomials ${(p_j)}_{1\leq j\leq B}$ based on  \Cref{th:groupmean} can be precomputed 
symbolically, e.g., 
with \textnormal{\texttt{SageMath}} using \textnormal{\texttt{sage.algebras.lie_algebras.bch}} that allows
 a polynomial base ring through the class
\begin{center} \textnormal{\texttt{sage.algebras.lie_algebras.lie_algebra.LieAlgebra}}
\end{center}
and thus a remarkably light-weighted implementation of \Cref{symbolic_bch}.
With the procedure \textnormal{\texttt{monomial_coefficient}} applied on our truncated Lie series, we can apply the dual basis to obtain $p_j$ for every $1\leq j\leq B$.
\end{remark}



\subsection{Reducing the number of terms with Gr\"obner bases}\label{sec:GBs}

Note that the family  of polynomials ${(p_j)}_{1\leq j\leq B}$ still contains redundancies, e.g., in  \Cref{ex:closedFormualsForL3d2} 
the expressions for $m_1$ and $m_2$ can be used to simplify the expression for $m_3$ via
\[
m_3 =\frac1N \sum_{i=1}^N  c_{3}^{(i)}    -\frac{m_1}2 \underbrace{\frac1{N} \sum_{i=1}^Nc_{2}^{(i)}}_{=m_2}{} + \frac{m_2}2 \underbrace{\frac1{N} \sum_{i=1}^Nc_{1}^{(i)}}_{=m_1}    = \frac1N \sum_{i=1}^N c_{3}^{(i)}.
\]
In what follows we show how to build a family of polynomials ${(r_j)}_{1\leq j\leq B}$
which yields the same coefficients $m_j$ as
in \Cref{symbolic_bch}, but with fewer terms. 
We use \DEF{Gr\"obner bases of modules} \cite{bib:MM1986,bib:KR2000} to find new representatives of those polynomial expressions (\ref{eq:mj}). \Cref{table:r_j} summarizes our findings for various values of $d$ and $L$. 

Let $q_j:={\mathcal{B}}^*_j\circ\BCH(-X,Y)$ with $X$ and $Y$ as in  \Cref{lem:pol_ex}, i.e., 
\begin{equation}\label{lemma:init_q}
q_j(m_1,\ldots, m_j, c_1,\ldots,c_j) = p_j(-m_1,\ldots,-m_{j-1},c_1,\ldots,c_{j-1})-m_j+c_j. 
\end{equation}
With this notation, \Cref{eq:mj} is equvalent to 
\begin{equation}\label{eq:mj_implicit}
\sum_{1\leq i\leq N} w^{(i)} q_j(m,c^{(i)})=0.
\end{equation}
For simplicity of presentation, in the remainder of the subsection, we assume  that $w^{(i)}=\frac1N$. 
Since we sum over all $1\leq i\leq N$, we allow only those reductions and S-polynomials, which are generated by the symbols $M_j$ in the following sense. 

\begin{lemma}\label{lem:overlapReduction}
Let $G\subseteq R$ such that for all $g\in G$,  
    $$\sum_{1\leq i\leq N}g(m,c^{(i)})=0.$$
\begin{enumerate}
        \item\label{lem:overlapReduction1}
       For all $q\in R$, $g\in G$, $\lambda \in \R$ and $1\leq{w_1},\dots, w_\ell\leq B$, the \DEF{reduction} 
       $$r=q-\lambda M_{w_1}\dots M_{w_\ell}g$$ evaluates to 
    $$\sum_{1\leq i\leq N}q(m,c^{(i)})=\sum_{1\leq i\leq N}r (m,c^{(i)}).$$
    \item\label{lem:overlapReduction2} For all $g,h\in G$,    $\lambda,\kappa \in\R$ and $1\leq{w_1},\dots, w_\ell,{v_1},\dots, v_t\leq B$, the \DEF{S-polynomial} 
    $$s=\lambda M_{w_1}\dots M_{w_\ell}g+\kappa M_{v_1}\dots M_{v_ t}h$$
    evaluates to 
    $$\sum_{1\leq i\leq N}s(m,c^{(i)})=0.$$
    \end{enumerate}
\end{lemma}
\begin{proof}
    Summation over $1\leq i\leq N$ and evaluation in $m$ and $c^{(i)}$ is linear. 
    Therefore we obtain \ref{lem:overlapReduction1} with  
    \begin{align*}
        \sum_{1\leq i\leq N}(\lambda M_{w_1}\dots M_{w_\ell}g) (m,c^{(i)})
        &=\sum_{1\leq i\leq N}\lambda m_{w_1}\dots m_{w_\ell}g(m,c^{(i)})\\
        &=
    \lambda m_{w_1}\dots m_{w_\ell}\sum_{1\leq i\leq N}g(m,c^{(i)})=0.
    \end{align*}
    Part \ref{lem:overlapReduction2} is similar. 
\end{proof}

For a fixed monomial ordering $<$ on $R$ let 
$\operatorname{lc},\operatorname{lm}$ and $\operatorname{lt}$ denote the \DEF{leading coefficient}, \DEF{leading monomial} and \DEF{leading term}, respectively.
 
The following \Cref{algo:normalForm} is a special case of the well-known reduction algorithm in Gr\"obner bases theory of modules (compare \Cref{rem:GBs_over_modules} for details on the reformulation to the more general framework).
It ``freezes'' symbols $C_j$ in the sense of \Cref{lem:overlapReduction}, and thus provides for all $q_j$ from (\ref{lemma:init_q}) new representatives modulo the previous $\{q_1,\dots,q_{j-1}\}$.

\LinesNotNumbered
\begin{algorithm2e}
\DontPrintSemicolon 
\KwIn{A polynomial $q\in R$, a set $G\subseteq R$ of polynomials, and a monomial ordering $<$ on $R$}
\KwOut{The reduced normal form $r={\text{\textbf{rNF}}}\bm{(q,G)}\in R$ of $q$}
\setcounter{AlgoLine}{0}
\nl $r\gets q$\;
\nl \If{$r=0$}{\nl\Return $0$}
\nl \While{\(\exists g\in G\,\exists w_i:\operatorname{lm}(r)=M_{w_1}\dots M_{w_\ell}\operatorname{lm}(g)\)}{
    \nl $r\gets r-\frac{\operatorname{lc}(r)}{\operatorname{lc}(g)}M_{w_1}\dots M_{w_\ell}g$\;
    \nl \If{$r=0$}{\nl\Return $0$}
    }
\nl \Return{ $\operatorname{lt}(r)+\operatorname{rNF}(r-\operatorname{lt}(r),G)$}\;
\caption{\textsc{Reduced normal form (with frozen $C$)}}
\label{algo:normalForm}
\end{algorithm2e}

Similarly, we can formulate Buchberger's algorithm (with frozen $C_j$). 
A detailed description is presented in the supplementary material.  
Therefore, we obtain for every $j\leq  B$ a Gr\"obner basis 
\begin{equation}\label{eq:Gj}
G^{(j)}:=\operatorname{Buchberger}(q_1,\ldots,q_{j})
\end{equation}
with respect to our input relations $\{q_1,\ldots,q_{j}\}$.

\begin{corollary}\label{cor:lexForGB}
    If $<$ is the lexicographic ordering with 
    linear preorder $$C_B>\cdots>C_1>M_B>\cdots>M_1,$$ then $G^{(j)}=\{q_1,\dots,q_j\}$ is already a Gr\"obner basis.
\end{corollary}
\begin{proof}
This is an immediate consequence of $\operatorname{lm}(q_j)=C_j$ and \Cref{lem:pol_ex}.
\end{proof}

\begin{theorem}
\label{lem:mod_quadratic_terms}
Let $<$ be arbitrary. 
\begin{enumerate}
\item\label{lem:r_facts1} 
For all $g\in G^{(j)}$,
$$\sum_{1\leq i\leq N}g(m,c^{(i)})=0.$$
\item\label{lem:r_facts2} 
$\operatorname{rNF}(q_j,G^{(j-1)})\in C_j-M_j+\R[M_1,\ldots,M_{j-1},C_1,\ldots,C_{j-1}].$
\end{enumerate}
\end{theorem}
\begin{proof}

A Gr\"obner basis can be considered as the closure of all S-polynomials according to 
\Cref{lem:overlapReduction}, i.e., 
    part \ref{lem:r_facts1} is valid.
    With \Cref{lemma:init_q} we have 
    $$\{q_1,\dots,q_{j-1}\}\subseteq \R[M_1,\ldots,M_{j-1},C_1,\ldots,C_{j-1}],$$
   thus  $G^{(j-1)}\subseteq \R[M_1,\ldots,M_{j-1},C_1,\ldots,C_{j-1}]$, and therefore part \ref{lem:r_facts2}, independent of the monomial ordering on $R$.
\end{proof}

\begin{corollary}\label{cor:rj_def}
In the general case of discrete measures (under assumptions of \Cref{th:groupmean}),
\Cref{eq:mj} \emph{simplifies} to 
\begin{equation}\label{eq:final_mj}
m_j=
\sum_{i=1}^N w^{(i)}\left(c_{j}^{(i)}+r_j(m_1,\dots,m_{j-1},c^{(i)}_1,\dots,c^{(i)}_{j-1})\right),
\end{equation}
where 
$ r_j:=\operatorname{rNF}(q_j,G^{(j-1)})-C_j+M_j$. 
Hence, in \Cref{algo:groupmean-grouplike}, the polynomials $p_j$ can be replaced with $r_j$, reducing the computational complexity of the algorithm.
\end{corollary}

\begin{table}
\begin{center}
\begin{small}
    \begin{tabular}{ c| c c c c c c c}
  $Q_{L,d}$ & $d=2$ & $d=3$ & $d=4$ & $d=5$ & $d=6$ & $\cdots$ &$d=10$\\\hline
  $L=2$ & $0$ & $0$ & $0$ & $0$ & $0$ &$\cdots$ & $0$\\
  $L=3$ & $2$ & $3$& $3$& $3$ & $3$ &$\cdots$ & $3$\\
  $L=4$ & $3$ & $\textbf{7}$ & $9$ & $9$ & $9$ & $\cdots$ & $9$\\
  $L=5$ & $15$ &$27$ & $36$ & $43$ & $43$ & $\cdots$ & $43$
\end{tabular}
\end{small}
\caption{We list the maximal number of extra terms $Q_{L,d}$ in $r_j$ for various $d$ and $L$ according to \Cref{cor:rj_def}.
In all cases except one, we use the ordering from \cref{cor:lexForGB}. However, in the constellation where $L=4$ and $d=3$ we use the degree lexicographic ordering with $C_1>\cdots>C_B>M_1>\cdots>M_B$, which produces $7$ terms. In this case, the default ordering from above produces $8$ terms.}
\label{table:r_j}
\end{center}
\end{table}

\begin{example}\label{ex:polys_final}
Continuing \Cref{ex:polys_by_pch}, we have 
  $$
\begin{bmatrix}q_1\\q_2\\q_3\\q_4\\q_5\end{bmatrix}=
\begin{bmatrix}C_{1} - M_{1}\\
C_{2} - M_{2}\\
-\frac{1}{2}C_{2}M_{1} + \frac{1}{2}C_{1}M_{2} + C_{3} - M_{3}\\
\frac{1}{12}C_{1}C_{2}M_{1} + \frac{1}{12}C_{2}M_{1}^2 - \frac{1}{12}C_{1}^2M_{2} - \frac{1}{12}C_{1}M_{1}M_{2} - \frac{1}{2}C_{3}M_{1} + \frac{1}{2}C_{1}M_{3} + C_{4} - M_{4}\\
-\frac{1}{12}C_{2}^2M_{1} + \frac{1}{12}C_{1}C_{2}M_{2} - \frac{1}{12}C_{2}M_{1}M_{2} + \frac{1}{12}C_{1}M_{2}^2 + \frac{1}{2}C_{3}M_{2} - \frac{1}{2}C_{2}M_{3} + C_{5} - M_{5}
\end{bmatrix}
$$
due to \Cref{lemma:init_q}. With $s_1=q_1$, $s_2=q_2$ and iterative reductions,  let
$$q_3 
\xrightarrow{q_1} -\frac{1}{2}C_{2}M_{1} + \frac{1}{2}M_{1}M_{2} + C_3 - M_{3} \xrightarrow{q_2} C_3-M_3=:s_3,
$$
\begin{align*}q_4
&\xrightarrow{q_1,q_2} \frac{1}{12}C_{1}C_{2}M_{1} - \frac{1}{12}C_{1}^2M_{2} - \frac{1}{2}C_{3}M_{1} + \frac{1}{2}C_{1}M_{3} + C_{4} - M_{4}\\
&\xrightarrow{q_1,s_3}\frac{1}{12}C_{1}C_{2}M_{1} - \frac{1}{12}C_{1}^2M_{2} + C_{4} - M_{4}=:s_4
\end{align*}
and 
\begin{align*}q_5
&\xrightarrow{q_1,q_2} -\frac{1}{12}C_{2}^2M_{1} + \frac{1}{12}C_{1}C_{2}M_{2} + \frac{1}{2}C_{3}M_{2} - \frac{1}{2}C_{2}M_{3} + C_{5} - M_{5}\\
&\xrightarrow{q_2,s_3} -\frac{1}{12}C_{2}^2M_{1} + \frac{1}{12}C_{1}C_{2}M_{2} + C_{5} - M_{5}=:s_5.
\end{align*} 
With this we replace $p\in R^5$ from \Cref{ex:polys_by_pch} with  ``shorter'' normal forms
$$r:=s+M-C=\begin{bmatrix}0&
0&
0&
\frac{1}{12}C_{1}C_{2}M_{1} - \frac{1}{12}C_{1}^2M_{2}&
-\frac{1}{12}C_{2}^2M_{1} + \frac{1}{12}C_{1}C_{2}M_{2}
\end{bmatrix}^\top\in R^5$$
from \Cref{cor:rj_def}. Evaluation of the latter improves the algorithm with  
\begin{align}
    m_1 &= \frac1N \sum_{i=1}^Nc_{1}^{(i)},
    \quad m_2 = \frac1N \sum_{i=1}^Nc_{2}^{(i)},\quad 
    m_3 = \frac1N \sum_{i=1}^N c_{3}^{(i)}  \nonumber\\
       m_4 &=\frac1N \sum_{i=1}^N\left(c_{4}^{(i)} + \frac{c_1^{(i)}}{12} (m_1 c_2^{(i)} - c_1^{(i)}m_2)\right),\quad
       m_5 =\frac1N \sum_{i=1}^N\left(c_5^{(i)} + \frac{c_2^{(i)}}{12}( c_1^{(i)}m_2 - m_1c_2^{(i)})\right).\nonumber
  \end{align}
\end{example}

A larger example involving nontrivial S-polynomials is provided in  the supplementary material.

\begin{remark}\label{rem:GBs_over_modules}
    We can reformulate \Cref{algo:normalForm} in the more general framework \cite{bib:MM1986,bib:KR2000}. For instance in  \Cref{ex:polys_final}   we consider our relations $q_j\in R$ as coefficient  vectors from the free module $\operatorname{span}_{\R[M_1,\dots,M_B]}(\mathcal{C})$ with basis $$\mathcal{C}:=\{1,C_1,C_2,C_3,C_4,C_5,C_1^2,C_1C_2,C_2^2\}.$$ 
    Hence, the relation $q_3$ can be considered as the polynomial coefficient vector  
$${\begin{bmatrix}- M_{3}&\frac{1}{2}M_{2}&-\frac{1}{2}M_{1}&1&0&0 &0&0&0
\end{bmatrix}}^\top,$$
or similarly for $q_4$,  
$${\begin{bmatrix}- M_{4}&
- \frac{1}{12}M_{1}M_{2} +\frac{1}{2}M_{3}&
 \frac{1}{12}M_{1}^2&
 - \frac{1}{2}M_{1}&
 1&
 0&
 - \frac{1}{12}M_{2} &
\frac{1}{12}M_{1} &   0
\end{bmatrix}}^\top.$$
With this, we are in the well-known  Gr\"obner theory of the module $\R[M_1,\dots,M_B]^9$.
\end{remark}

\subsection{Reducing the number of terms with an antisymmetrized BCH formula}\label{sec:BCH-like_cancl}
%
%
In this section, we delve into the cancelations that emerge from the Baker-Campbell-Hausdorff formula and terms linear in $C_j$. These explicit cancelations result in polynomials with a comparable number of terms as in the previous section (\Cref{rem:comparison}).

For this, we use the 
Baker-Campbell-Hausdorff formula with graded components  
\begin{equation}\label{eq:BCH}
\log( \exp(X) \exp(Y) ) = \BCH(X,Y) = \sum\limits_{k=1}^{\infty} \BCH_{k}(X,Y)
\end{equation}
expanded in the Lyndon basis (\cite[Table 2, page 12]{bib:CM2009} or \cite{bib:CMTables}) of the free Lie algebra over the two-letter alphabet $\{X,Y\}$.
The first $6$ components $\BCH_k (X,Y)$  are given by:
\begin{align*}
\BCH_{1}(X,Y)  &= X+Y, \qquad\qquad\qquad\qquad\qquad\quad\ \BCH_{2}(X,Y) = \frac{1}{2}[X,Y],  \\
\BCH_{3}(X,Y) &= \frac{1}{12}[[X,Y],Y] + \frac{1}{12}[X,[X,Y]],
    \qquad \BCH_{4}(X,Y)  = \frac{1}{24}[X,[[X,Y],Y]] , \\
\BCH_{5}(X,Y)  &=  -\frac{1}{720}[[[[X,Y],Y],Y],Y] +
\frac{1}{360}[[X,[X,Y]], [X,Y]]  + \frac{1}{120}[[X,Y],[[X,Y],Y]] \\
& +  \frac{1}{180}[X,[[[X,Y],Y],Y]] +
\frac{1}{180}[X,[X,[[X,Y],Y]]]  - \frac{1}{720}[X,[X,[X,[X,Y]]]], \\
\BCH_{6}(X,Y) &=   -\frac{1}{1440}[X,[[[[X,Y],Y],Y],Y]] +
\frac{1}{720}[X,[[X,[X,Y]], [X,Y]]] \\
& + \frac{1}{240}[X,[[X,Y],[[X,Y],Y]]] 
 +  \frac{1}{360}[X,[X,[[[X,Y],Y],Y]]] \\
 &-
\frac{1}{1440}[X,[X,[X,[[X,Y],Y]]]].
\end{align*}
\begin{remark}[{\cite[Section IV.C]{bib:CM2009}}]\label{rem:BCH_even_order}
For even $k$  all the terms in  $\BCH_k(X,Y)$ have necessarily the form $[X,Z]$, where $Z$ goes over all Lyndon basis elements of degree $k-1$, 
except for the element $[X,[X,\ldots,[X,Y]]]$ which does not appear. 
\end{remark}
For such Lyndon words as in \Cref{rem:BCH_even_order}, we define a map 
\[
\DX([X,Z]) := Z.
\]
This helps us to introduce an asymmetrized Baker-Campbell-Hausdorff formula $\aBCH(X,Y)$ as follows.
\begin{lemma}\label{lem:aBCH}
The Lie series   
\begin{equation}\label{eq:barycenter_recursive}
\aBCH(X,Y):= \sum\limits_{k=1}^{\infty} \aBCH_k(X,Y):=\DX(\BCH(X,Y) - \BCH(Y,X))
\end{equation}
is well-defined and has graded components
\begin{equation}\label{eq:SK}
\aBCH_{k} (X,Y) = 
\begin{cases}
0, & \text{ if }k \text{ is even}, \\
2\,\DX(\BCH_{k+1}(X,Y)), & \text{ if }k \text{ is odd};
\end{cases}
\end{equation}
\end{lemma}
for example, the first $3$ nonzero graded components are given by
\begin{align*}
&\aBCH_{1}(X,Y)  =Y,\quad \aBCH_{3}(X,Y))=\frac{1}{12}[[X,Y],Y], \\
&\aBCH_{5}(X,Y) =   -\frac{1}{720}[[[[X,Y],Y],Y],Y] +
\frac{1}{360}[[X,[X,Y]], [X,Y]] \\
& + \frac{1}{120}[[X,Y],[[X,Y],Y]] +  \frac{1}{180}[X,[[[X,Y],Y],Y]] -
\frac{1}{720}[X,[X,[[X,Y],Y]]].
\end{align*}
\begin{proof}[Proof of \Cref{lem:aBCH}]
From \eqref{eq:bch_permutation}, we get that
$$
\BCH(X,Y) - \BCH(Y,X) = \BCH(X,Y) + \BCH(-X,-Y).
$$
Splitting the equation by degrees and using the fact that 
\[
\BCH_{k+1}(-X,-Y) = (-1)^{k+1}\BCH_{k+1}(X,Y),
\]
we get 
\[
\BCH_{k+1}(X,Y) - \BCH_{k+1}(Y,X) = 
\begin{cases}
0, & k \text{ is even},\\
2\BCH_{k+1}(X,Y), & k \text{ is odd}. 
\end{cases}
\]
Finally, by \Cref{rem:BCH_even_order} the operator $\DX$ can be applied to non-zero terms, which proves \Cref{eq:SK}.
\end{proof}

As we will show next, the asymmetrized version of the BCH formula can be used in \Cref{lem:pol_ex}.

\begin{theorem}\label{thm:aBCH}
\begin{enumerate}
\item 
For the discrete measure supported on ${\mathbf{x}}^{(i)}$, $i\in \{1,\ldots,N\}$ (as in \Cref{th:groupmean}), let $C^{(i)} = \log(\mathbf{x^{(i)}})$. 
Then the logarithm of the barycenter  $M = \log(\mathbf{m})$ satisfies 
\begin{equation}\label{eq:barycenter_ABCH}
M = \sum\limits_{i=1}^{N} w^{(i)} \aBCH(-M, C^{(i)}).
\end{equation}
\item
Let $m_j$ and $c_j^{(i)}$, $j \in \{1, \ldots, B\}$ are the coordinates of $\mathbf{m}$ and $\mathbf{x}^{(i)}$ in the  basis $\mathcal{B}$.
Then $m_j$ can be computed recursively similarly to 
\Cref{eq:mj}, i.e., 
\begin{equation}\label{eq:final_mj_BCHcancl}
m_j=
\sum_{i=1}^N w^{(i)}\left(c_{j}^{(i)}+r_j(m_1,\cdots,m_{j-1},c^{(i)}_1,\cdots,c^{(i)}_{j-1})\right).
\end{equation}
where the polynomial $r_j$ is defined for $X$ and $Y$, expanded as in \Cref{lem:pol_ex}, as follows
\[
r_j:={\mathcal{B}}^*_j\left(\sum_{\substack{k=3,\dots,L\\k \text{ odd}}} \aBCH_k(-X,Y)\right)\in R.
\]
\end{enumerate}
\end{theorem}

\begin{remark}\label{rem:comparison}
For all truncation levels $L\leq 5$, the maximal number of terms in \Cref{thm:aBCH} coincides with those provided via Gr\"obner reductions (with respect to the ordering from \cref{cor:lexForGB}). This is illustrated by \Cref{table:r_j}. 
There is one case when the maximal number of terms is higher ($8$ instead of $7$  terms for $d=3$ and $L=4$). Details are provided in the supplementary materials. Note that the evaluations in \Cref{eq:final_mj,eq:final_mj_BCHcancl} are identical, and only the relations $r_j$ are constructed in a different way.  
\end{remark}


The proof of \Cref{thm:aBCH} relies on some useful facts from \cite{bib:R1993} which we recall below.
We denote  $\ad_{X}(Z) := [X,Z]$ and the  ``power'' of $\ad_X$
\[
    \ad^{k}_{X}(Y) = \overbrace{[X,[X,\ldots, [X}^{k \text{ times}},Y].
\]
The following lemma is key for this subsection.

\begin{lemma}[{\cite[page 80]{bib:R1993}}]\label{lem:linear}
We expand the Baker-Campbell-Hausdorff formula as follows:
\begin{equation}\label{eq:BCH_powers}
\BCH(X,Y) = X + \underbrace{H_1(Y)}_{\text{linear in } Y} + \underbrace{H_{\ge 2}(Y)}_{\text{order }\ge 2\text{ terms in $Y$ }}.
\end{equation}
Then the linear can be expressed as 
\[
H_1(Y) = Y + \frac{1}{2}[X,Y] + \sum\limits_{n=1}^{\infty}\frac{b_{2n}}{(2n)!} \ad^{2n}_{X}(Y)  
\]
where $b_{2n}$ are Bernoulli numbers. Alternatively, we can write
\[
H_1(Y) = g(\ad_{X})(Y),
\]
where $g(\ad_{X})$ means substitution of $t^k$ with $\ad^k_X$ in the power series expansion
\[
g(t) = \frac{t}{1-e^{-t}} = 1+g_1 t + g_2 t^{2} + \cdots.
\]
\end{lemma}
\begin{remark}\label{rem:H1B}
\begin{enumerate}
    \item 
With $\ad_{X}(X)=0$, $H_1$ leaves $X$ invariant, i.e., $H_1(X) = X$. 
\item\label{rem:H1inverse}
$H_1$ is invertible, and its inverse is
$H^{-1}_1(Z) = f (\ad_X) (Z)$,
where 
\[
f(t) = \frac{1}{g(t)} = \frac{1-e^{-t}}{t}.
\]
\end{enumerate}
\end{remark}



\begin{proof}[Proof of \Cref{thm:aBCH}] Part 2 follows directly from part 1, therefore we are going to prove part 1.
First, we show that 
\[
\aBCH(X,Y) = H_1^{-1}((\BCH(X,Y)-X)
\]
Indeed, by \Cref{rem:H1B},  we have

\begin{align*}
&H_1^{-1}((\BCH(X,Y)-X)  =  \frac{1 - e^{-\ad_X}}{\ad_X} (\log(e^Xe^Y) - X) \\
    &= \DX \left( \log(e^Xe^Y) - X - e^{-\ad_X}  (\log(e^Xe^Y) - Y)\right)\\
  &=\DX \left( \log(e^Xe^Y) - X - (\log(e^Ye^X) - X)\right) =  \DX (\BCH(X,Y) - \BCH(Y,X)),
\end{align*}   
where for the last but one equality we use the well-known property
\[
    e^{\ad_A}(Z) = e^{A} Z e^{-A},
\]
to get
\begin{equation}\label{eq:adB_inverse_application}
e^{-\ad_X} (\log(e^Xe^Y) - X) = e^{-X}\log(e^Xe^Y)e^X - e^{-X} X e^X= \log(e^{Y}e^{X}) - X.
\end{equation}
Recall $M$ is the logarithm of the barycenter iff
\begin{equation}\label{eq:BCH_reminder}
\sum\limits_{i=1}^N w^{(i)} \BCH(-M,C^{(i)}) = 0.
\end{equation}
The main idea is to eliminate linear terms in the BCH formula.
Using the splitting of the BCH, linearity of $H_1$  and the fact that weights $w^{(i)}$ sum to one, we get that
\begin{align}
\sum\limits_{i=1}^N w^{(i)} \BCH(-M,C^{(i)})
& = \sum\limits_{i=1}^N w^{(i)}( -M + H_1(C^{(i)}) + H_2(C^{(i)})) \nonumber\\
&= 
H_1( -M + \sum\limits_{i=1}^N w^{(i)} C^{(i)} ) + \sum\limits_{i=1}^N w^{(i)}H_2(C^{(i)}).\label{eq:carycenter_rewriting_H1}
\end{align}
Therefore, \eqref{eq:BCH_reminder} is equivalent to
\[
 M - \sum\limits_{i=1}^N w^{(i)} C^{(i)} =  \sum\limits_{i=1}^N w^{(i)} H_{-1}(H_2(C^{(i)})).
\]
Finally, we note that
\[
Y + H_1^{-1}(H_{\ge 2}(Y)) = H_1^{-1}( H_1(Y) + H_{\ge 2}(Y)) = H^{(-1)} (\BCH(X,Y)-X)
\]
hence \eqref{eq:BCH_reminder} is equivalent to \eqref{eq:barycenter_ABCH}.
\end{proof}

\subsection{Recursive updates of group means}
The following formula allows an update of a computed mean if a new data point is inserted. 
Use Taylor expansion and obtain the following online update formula.
\begin{lemma}
 Let \(\mathbf{m}'\) be the group mean of \(N-1\) points \((\mathbf{x}^{(i)})_{1\leq i\leq N-1}\) and \(\mathbf{m}\) be the group mean of those same \(N-1\) points plus an additional incoming point \(\mathbf{x}^{(N)}\). The computation of \(\mathbf{m}\) can be derived from the value of \(\mathbf{m}'\) via 
\begin{multline} \label{eq:upd}
    m_j = \frac1N {\tilde p}_j(m_1,\dots,m_{j-1},c^{(N)}_{1},\dots, c^{(N)}_{j}) + \frac{N-1}{N}m'_j \\
    + \frac1N \sum_{i=1}^{N-1} \sum_{1\leq \lvert\alpha\rvert\leq j-1} \frac{(\Delta m)^{\alpha}}{\alpha !} \frac{\partial^{\alpha}{\tilde p}_j(m_1',\dots,m_{j-1}',c_{1}^{(i)},\dots,c_{j}^{(i)})}{(\partial m)^{\alpha}}, 
\end{multline} where $\tilde p_j:=p_j+C_j$ with $p_j$ due to  \Cref{lem:pol_ex} and $\Delta m := m-m'$. Equation~\eqref{eq:upd} requires to compute at most \(\binom{j-1+\deg(p_j)}{j-1}=\binom{j-1+\deg(p_j)}{\deg(p_j)}\) partial derivatives.
\end{lemma}
\begin{proof} a 
    \begin{align*}
        m_j &= \frac1N \sum_{i=1}^N \tilde p_j (m_1,\dots, m_{j-1},c_1^{(i)},\dots,c_j^{(i)})\\
        &= \frac1N \tilde p_j (m_1,\dots, m_{j-1},c_1^{(N)},\dots,c_j^{(N)}) + \frac1N \sum_{i=1}^{N-1} \tilde p_j (m_1,\dots, m_{j-1},c_1^{(i)},\dots,c_j^{(i)}).
    \end{align*}
    Now apply Taylor expansion to the term inside the sum. We have 
    \begin{multline*}
        \tilde p_j (m'_1 + (\Delta m)_1,\dots, m'_{j-1}+(\Delta m)_{j-1},c_1^{(i)},\dots,c_j^{(i)}) = \tilde p_j (m'_1,\dots, m'_{j-1},c_1^{(i)},\dots,c_j^{(i)}) \\+ \sum_{1\leq \lvert \alpha \rvert \leq j-1} \frac{(\Delta m)^{\alpha}}{\alpha !}\frac{\partial^{\alpha}\tilde p_j(m_1',\dots,m_{j-1}',c_{1}^{(i)},\dots,c_{j}^{(i)})}{(\partial m)^{\alpha}}.
    \end{multline*}
    Inject that back into $m_j$ to obtain Equation~\eqref{eq:upd}.
\end{proof}

\section{Algorithm using  updates in the ambient space}\label{sec:ambient}

In this section, we leverage the embedding of the space of signatures in the truncated tensor algebra $T_{\leq L}(\mathbb{R}^d)$. We develop an iterative procedure for obtaining the successive levels of the group mean in terms of lower levels, that is $m_K\in(\mathbb{R}^d)^{\otimes K}$ as a mapping of $(m_0,\dots,m_{K-1})$ and the data. This mapping involves two polynomials $p$ and $q$ that are obtained from the truncated version of the logarithm mapping in $T_{\leq L}(\mathbb{R}^d)$. 

Notably, this approach is applicable to any dimension $d$. It also does not need to compute the polynomials $p_j$ or $r_j$ (unlike the algorithms in \Cref{sec:barycenterSymbolic}), at the expense of dealing with higher-dimensional tensor spaces. The corresponding algorithm essentially relies on the computation of tensor products and we provide a Python implementation.

First, we present the result and the corresponding algorithm. Then, examples illustrate the proof. Finally, we look at time and memory complexities.

\noindent\textbf{Notation.} Throughout this section, lower indices of group elements $\mathbf{g}\in\mathcal{G}_{\leq L}(\mathbb{R}^d)$ denote levels: $\mathbf{g}=(\mathbf{g}_0,\dots, \mathbf{g}_L)$.

\subsection{Main result}

We have a truncated version (by nilpotency) of the inverse and of the logarithm mapping presented in Section~\ref{sec:background}. For any \((\e+\mathbf{g})\in \G_{\leq L}(\mathbb{R}^d)\), \begin{equation}\label{eq:inversionandlog}
    (\e+\mathbf{g})^{-1} := \sum_{k=0}^L (-1)^k \mathbf{g}^k , \qquad \log(\e+\mathbf{g}) :=  \sum_{k=1}^L \frac{(-1)^{k+1}}{k}\mathbf{g}^k .
\end{equation}
Recall that the group operation in tensor algebra for 
\[
\mathbf{a} = (\mathbf{a}_0, \mathbf{a}_{1}, \ldots \mathbf{a}_L), \quad
\mathbf{x} = (\mathbf{x}_0, \mathbf{x}_{1}, \ldots \mathbf{x}_L)
\]
(note that $\mathbf{a}_0 = \mathbf{x}_0 = 1$ for $\mathbf{a},\mathbf{x} \in G$)is given by the Chen's identity
\begin{equation}\label{eq:chen}
(\mathbf{a} \mathbf{b})_{L} =  \sum\limits_{\ell=0}^{L} \mathbf{a}_{\ell} \otimes  \mathbf{x}_{L-\ell}.
\end{equation}

\begin{proposition}\label{prop:1}
Let \(\{X^{(1)},\dots,X^{(N)}\}\) be a batch of \(N\) multivariate time series with \(d\) components. Let $\{w^{(1)},\dots, w^{(N)}\}$ be a set of real values such that $\sum_{i=1}^N w^{(i)} = 1$. Denote \(\mathbf{m}\) the group mean of this dataset with weights $w^{(i)}$. Denote $\mathbf{a}:=\mathbf{m}^{-1}$ and $\mathbf{x}^{(i)}:=\mathsf{IIS}_{\leq L}(X^{(i)})$. One has \begin{equation}
    \mathbf{a}_1=-\sum_{i=1}^N w^{(i)}(\mathbf{x}^{(i)})_1
\end{equation} and for any $K=2,\dots,L$,
\begin{multline}\label{eq:mean2}
    \mathbf{a}_K=-\sum_{i=1}^Nw^{(i)}\bigg(q_K\left(\mathbf{a}_{0},\dots,\mathbf{a}_{K-1}, (\mathbf{x}^{(i)})_1,\dots,(\mathbf{x}^{(i)})_K\right)\\+p_K\left(\mathbf{a}_0,\dots,\mathbf{a}_{K-1},(\mathbf{x}^{(i)})_1,\dots, (\mathbf{x}^{(i)})_{K-1}\right) \bigg)
\end{multline} where $p_K$ and $q_K$ are (noncommutative) polynomial functions.
\end{proposition}
Note that the computation on the right-hand side solely relies on the values of $\mathbf{a}_0, \dots, \mathbf{a}_{K-1}$ and the input data $\{\mathbf{x}^{(1)}, \dots,\mathbf{x}^{(N)}\}$. In other words, this feature enables an iterative approach for calculating the successive values of $\mathbf{a}_K$ for increasing values of $K$.

Before proving \Cref{prop:1}, we begin with a preliminary consideration. 

\begin{lemma}\label{lem:vizero} 
For any \(0\leq K<j\) and \(1\leq i\leq N\), denote $v^{(i)}:=\mathbf{a} \mathbf{x}^{(i)} -\e$ and   $v^{(i,j)}:=(v^{(i)})^j$. Then, 
    \begin{equation}
        (v^{(i,j)})_{K} = 0 .
    \end{equation}
\end{lemma}
\begin{proof}
     By induction on \(j\). We have \((v^{(i)})_0 = (\mathbf{a} \mathbf{x}^{(i)})_0-1=0\). Now, suppose that for a fixed \(j\) we have \((v^{(i,j-1)})_{0}=\dots=(v^{(i,j-1)})_{j-2}=0\).  
     
     If \(K<j-1\) then, from \eqref{eq:chen} 
\begin{equation*}
        (v^{(i,j)})_K = (v^{(i)} v^{(i,j-1)})_K
         = \sum_{k=0}^K (v^{(i)})_{K-k}\otimes (v^{(i,j-1)})_{k}
         = 0
     \end{equation*} using the induction hypothesis. 
     
     If \(K=j-1\) then
     \begin{equation*}
         (v^{(i, j)})_K = \sum_{k=0}^K (v^{(i)})_{K-k}\otimes (v^{(i, j-1)})_k= (v^{(i)})_0\otimes (v^{(i, j-1)})_{j-1}= 0
     \end{equation*} using the induction hypothesis and that \((v^{(i)})_0=0\).
\end{proof}

\begin{proof}[Proof of \Cref{prop:1}]
 Using Definition~\ref{def:groupmean}, the group mean $\mathbf{m}$ with weights $w^{(i)}$ verifies \begin{equation}\label{eq:mean}
    0 = \sum_{i=1}^N w^{(i)} \log(\mathbf{m}^{-1} \mathbf{x}^{(i)})=\sum_{i=1}^N w^{(i)}\log(\e+v^{(i)}) = \sum_{i=1}^N \bigg(w^{(i)}v^{(i)} + w^{(i)}\sum_{j=2}^L \frac{(-1)^{j+1}}{j} v^{(i,j)}\bigg)
\end{equation}
where the last equality is obtained using the definition of the logarithm mapping (Equation~\eqref{eq:inversionandlog}). Denote \(z\) the last right-hand side of \Cref{eq:mean}. Using \Cref{lem:vizero}, when $z$ is evaluated at level $1$, we have \begin{equation*}
    z_1 = \sum_{i=1}^N w^{(i)} (v^{(i)})_1 = \sum_{i=1}^N w^{(i)} (\mathbf{a}_1 + (\mathbf{x}^{(i)})_1) = 0,
\end{equation*} therefore $\mathbf{a}_1 = -\sum_{i=1}^N w^{(i)}\mathbf{x}^{(i)}$.  
Using \Cref{lem:vizero}, for any $K=2,\dots,L$, when $z$ is evaluated at level $K$, the sum stops at $K$: 
\begin{equation}\label{eq:zk}
    z_K = \sum_{i=1}^N\left(w^{(i)}(v^{(i)})_K + w^{(i)} \sum_{j=2}^K \frac{(-1)^{j+1}}{j} (v^{(i,j)})_K\right) =0 .
\end{equation} Now remark that \(\sum_{j=2}^K \frac{(-1)^{j+1}}{j} (v^{(i,j)})_K\) depends only on \((v^{(i)})_1,\dots,(v^{(i)})_{K-1}\). Therefore, we can denote 
\begin{equation}\label{eq:pk}
    p_K\left((v^{(i)})_1,\dots,(v^{(i)})_{K-1}\right) := \sum_{j=2}^K \frac{(-1)^{j+1}}{j} (v^{(i,j)})_K .
\end{equation}
Also, from the definition of $v_i$, and from Chen's identity \eqref{eq:chen}
\begin{equation}\label{eq:vi}
    (v^{(i)})_K:=(\mathbf{a} \mathbf{x}^{(i)} -\e)_K =\mathbf{a}_K+\sum_{k=0}^{K-1} \mathbf{a}_k\otimes (\mathbf{x}^{(i)})_{K-k} .
\end{equation} Denote  \begin{equation}\label{eq:qk}
    q_K\left(\mathbf{a}_0,\dots \mathbf{a}_{K-1}, (\mathbf{x}^{(i)})_1,\dots,(\mathbf{x}^{(i)})_{K}\right) := \sum_{k=0}^{K-1}\mathbf{a}_k\otimes (\mathbf{x}^{(i)})_{K-k} .
\end{equation}
Combining Equation~\eqref{eq:zk} and Equation~\eqref{eq:vi} gives~\eqref{eq:mean2}. 
\end{proof}

\subsection{Algorithm}
Let $(\alpha_0,\alpha_1,\dots)$ be real values such that \(\alpha_0:=0\) and \(\alpha_k:=\sum_{i=0}^{k-1}d^{i}\) for any \(k\geq1\). \(\alpha_{k+1}\) corresponds to the number of real values contained in an element of \(\G_{\leq k}(\mathbb{R}^d)\). For efficiency, the group elements of \(\G_{\leq L}(\mathbb{R}^d)\) are implemented as a long array of size \(\alpha_{L+1}\). The algorithm is detailed in Algorithm~\ref{algo:group-mean-ambient} with corresponding nomenclature shown in Table~\ref{tab:nomenclature}. Then, we derive the corresponding time and space complexities in \Cref{pro:complexities}.
\begin{table}
\centering
\begin{tabular}{c|c|c|c}\hline
Symbol      & Meaning   & Tensor order & Size     \\ \hline
$\mathbf{x}^{(i)}$ & Signatures of input time series $X^{(i)}$ & 2       & $N\alpha_{L+1}$  \\
$\mathbf{m}$ & Group mean  & 1    & $\alpha_{L+1}$   \\
$\mathbf{a}$ & Group inverse of $\mathbf{m}$    & 1            & $\alpha_{L+1}$   \\
$p^{(i)}$ & Polynomial of \Cref{eq:pk}  & 2            & $N\alpha_{L+1}$  \\
$q^{(i)}$ & Polynomial of \Cref{eq:qk}   & 2            & $N\alpha_{L+1}$  \\
$v^{(i)}$ & $\mathbf{a} \mathbf{x}^{(i)} -\mathsf{e}$  & 2  & $N\alpha_{L+1}$ \\
$v^{(i,j)}$ & Powers $(v^{(i)})^j$ for $j=2,\dots,L$  & 3  & $N(L-1)\alpha_{L+1}$ \\
$w^{(i)}$ & Weights of the group mean & 1            & $N$              \\ \hline
\end{tabular}
\caption{Nomenclature for tensors in Algorithm~\ref{algo:group-mean-ambient}. Index $i$ varies between $1$ and $N$.}
\label{tab:nomenclature}
\end{table}
\LinesNotNumbered
\begin{algorithm2e}\label{algo:group-mean-ambient}
\DontPrintSemicolon 
\KwIn{A batch of $N$ signatures \(\mathbf{x}^{(i)} := \mathsf{IIS}_{\leq L}(X^{(i)})\) and \(N\) weights $w^{(i)}$.}
\KwOut{An array $\mathbf{m}$ with the values of the group mean.}
\SetKwFunction{GroupInv}{GroupInv}
\setcounter{AlgoLine}{0}
\nl Initialize $\mathbf{a}_0 \gets 1,\; \mathbf{a}_1 \gets -\sum_{i=1}^N w^{(i)}(\mathbf{x}^{(i)})_1$ and $q_1^{(i)}\gets (\mathbf{x}^{(i)})_1$.\;
 \nl \For{$K=2,\dots, L$}{
\nl \For{$i=1,\dots,N$}{
\nl   $(v^{(i)})_{K-1} \gets \mathbf{a}_{K-1} + q^{(i)}_{K-1}$\;
\nl $q^{(i)}_K \gets q_K(\mathbf{a}_0, \dots, \mathbf{a}_{K-1},(\mathbf{x}^{(i)})_1,\dots, (\mathbf{x}^{(i)})_K)$\;
\nl \For{$j=2,\dots,K$}{
\nl $(v^{(i,j)})_K\gets \sum_{k=j-1}^K (v^{(i,j-1)})_k\otimes (v^{(i)})_{K-k}$\;
}
\nl $p^{(i)}_K \gets \sum_{j=2}^K \frac{(-1)^{j+1}}{j} (v^{(i,j)})_K$\;
}
\nl $\mathbf{a}_K \gets -\sum_{i=1}^N w_i \left(q_K^{(i)} + p_K^{(i)}\right)$\;
}
\nl $\mathbf{m} \gets $\GroupInv{$\mathbf{a}$}\tcp*{Computed by implementing \Cref{eq:inversionandlog}}
\nl \Return{\(\mathbf{m}\)}\;
\caption{\textsc{Group mean using updates in ambient space}}
\end{algorithm2e}

\begin{proposition}\label{pro:complexities}
    Time complexity of Algorithm~\ref{algo:group-mean-ambient} is \(\mathcal{O}(Nd^{L-1}L)\) and space complexity is \(\mathcal{O}(Nd^LL)\).
\end{proposition}
\begin{proof} Computation of \(q_K\) in Equation~\eqref{eq:mean2} is \(\mathcal{O}(K-1)d^K\). Computation of  \(p_K\) in Equation~\eqref{eq:mean2} is \(\mathcal{O}(d^{L-1}L)\) for a fixed \(i\), since computation of \((v^{(i,k)})_l\) is \(\mathcal{O}(d^{L-1}L)\). Indeed, let us fix  \(1\leq k\leq l\leq L\) (\(k>l\) is covered in \Cref{lem:vizero}). The computation of \((v^{(i,k)})_l\) requires \((l-k+1)d^{l-k+1}\) operations using that \begin{equation*}
        (v^{(i,k)})_l = (v^{(i)} v^{(i,k-1)})_l = \sum_{j=1}^{l-(k-1)} (v^{(i)})_j\otimes (v^{(i,k-1)})_{l-(k-1)-j}
    \end{equation*} since \(v^{(i,k-1)}\) has \(k-1\) leading zeros. Now, the computation for any \(1\leq k\leq l\leq L\) is \(\mathcal{O}(d^{L-1}L)\). To obtain the time complexity of Algorithm~\ref{algo:group-mean-ambient}, we have to take into account the batch size \(N\) and we obtain \(\mathcal{O}(Nd^{L-1}L)\).  

    Regarding space complexity, we have to store in memory \(v^{(i,K)}\) for all observation indices \(1\leq i\leq N\) and powers \(1\leq K\leq L\). Thus, the space complexity is \(\mathcal{O}(Nd^L L)\).
\end{proof}

\begin{remark} 
    In practical applications such as the analysis of a set of time series, the computation of the signature must be taken into account, especially when benchmarking against other methods. Let \(X:[a,b]\to\mathbb{R}^d\) be a linear process, observe that \begin{equation}\label{eq:siglinear}
        \mathsf{IIS}_{\leq L}(X) = \exp \left(X(b)-X(a)\right) .
    \end{equation} Now consider $X:[0,\TIME]:\to\mathbb{R}^d$ to be a piecewise linear process where each piece is defined on intervals $[t,t+1]$ with $t$ integer. Using Equation~\eqref{eq:siglinear} and Chen's identity, the signature can be computed iteratively:  \begin{equation}
        \mathsf{IIS}_{\leq L}(X) = \exp(X(2)-X(1)) \exp(X(3)-X(2)) \dots  \exp(X(\TIME)-X(\TIME-1)) .
    \end{equation} The product operation \(A e^Z\) is \(\mathcal{O}(d^L)\). Thus, the time complexity of the signature computation is \(\mathcal{O}(\TIME d^L)\). Combining this with the complexity of Algorithm~\ref{algo:group-mean-ambient}, the overall complexity of the approach is \(\mathcal{O}(Nd^L(\TIME+L))\). 
\end{remark}

\subsection{Examples}
To have a better grasp of the idea behind the algorithm, we show here the computations for the first two levels $L=1,2$. As stated before, the computation stands for any dimension $d$.

\begin{example}
At level 1, starting from the value of $\mathbf{a}_1=-\sum_{i=1}^N w^{(i)}(\mathbf{x}^{(i)})_1 $ computed in the proof above and since $\mathbf{g}\in \mathcal{G}_{\leq L}(\mathbb{R}^d)$, $(\log \mathbf{g})_1 = \mathbf{g}_1$ and \Cref{eq:inversionandlog}, we obtain
\begin{equation}\label{eq:m_ambient_L1}
    (\log \mathbf{m})_1 = \sum_{i=1}^N w^{(i)}(\log \mathbf{x}^{(i)})_1 .
\end{equation}
\end{example}
\begin{example}
At level 2, we have,
\begin{align*}
    \mathbf{a}_2 &= -\sum_{i=1}^N w^{(i)} \left(q_2\Big(\mathbf{a}_0,\mathbf{a}_1,(\mathbf{x}^{(i)})_1, (\mathbf{x}^{(i)})_2\Big) + p_2\Big((v^{(i)})_1\Big)\right)\\
    &= - \sum_{i=1}^N w^{(i)}\left((\mathbf{x}^{(i)})_2+\mathbf{a}_1\otimes (\mathbf{x}^{(i)})_1-\frac12 (v^{(i)})_1\otimes(v^{(i)})_1\right)\\
    &= -\sum_{i=1}^N w^{(i)} (\mathbf{x}^{(i)})_2 - \mathbf{a}_1\otimes \mathbf{a}_1 + \frac12 \sum_{i=1}^N w^{(i)}(\mathbf{a}_1+(\mathbf{x}^{(i)})_1)\otimes (\mathbf{a}_1+(\mathbf{x}^{(i)})_1) \\
&= \sum_{i=1}^N w^{(i)}\left(\frac12 (\mathbf{x}^{(i)})_1\otimes (\mathbf{x}^{(i)})_1 - (\mathbf{x}^{(i)})_2\right) +\frac12 \mathbf{a}_1\otimes \mathbf{a}_1,
\end{align*}
where  the last equality follows from \eqref{eq:m_ambient_L1}.
Using the fact that for any $\mathbf{g}\in \mathcal{G}_{\leq L}(\mathbb{R}^d)$, $(\log \mathbf{g})_2 = \mathbf{g}_2-\frac12 \mathbf{g}_1\otimes \mathbf{g}_1$ and \Cref{eq:inversionandlog}, we have
\begin{equation}\label{eq:m_ambient_L2}
    (\log \mathbf{m})_2 = \sum_{i=1}^N w^{(i)}(\log \mathbf{x}^{(i)})_2 .
\end{equation}
\end{example}

\begin{remark}
    As we have seen using the BCH formula in \Cref{ex:closedFormualsForL3d2}, the first two levels of $\log \mathbf{m}$ correspond to the Euclidean mean of $(\log \mathbf{x}^{(i)})_{i=1,\dots,N}$.
\end{remark}

\subsection{Expressions in the ambient space using the asymmetrized BCH formula}
We conclude this section by noting that we can also find the explicit expressions in the ambient space using the asymmetrized BCH formula developed in \Cref{sec:BCH-like_cancl}.
With some abuse of notation, we denote by $\mathbf{b} = \log (\mathbf{m}^{-1})$ and $\mathbf{c}^{(i)} = \log (\mathbf{x}^{(i)})$, and we view them as elements of tensor algebra, split them by orders:
\begin{equation}\label{eq:tensor_coefficient_barycenter}
\begin{split}
&\mathbf{b} = (0, \mathbf{b}_1, \mathbf{b}_2, \ldots, \mathbf{b}_L, \ldots) \in T(( \R^ d )),\\
&\mathbf{c}^{(i)} = (0, \mathbf{c}^{(i)}_1, \mathbf{c}^{(i)}_2, \ldots, \mathbf{c}^{(i)}_L, \ldots) \in T(( \R^ d )). 
\end{split}
\end{equation}
Next, we recall  that Lie brackets in the tensor space can be computed, for $\mathbf{u}_{k} \in (\mathbb{R}^{d})^{k}$ and $\mathbf{v}_{\ell} \in (\mathbb{R}^{d})^{\ell}$, as
\[
[\mathbf{u}_{k},\mathbf{v}_{\ell}] := \mathbf{u}_{k} \otimes \mathbf{v}_{\ell} - \mathbf{v}_{\ell} \otimes \mathbf{u}_{k} \in (\mathbb{R}^{d})^{\otimes(k+\ell)}.
\]
If $\mathbf{x},\mathbf{y} \in T(( \R^ d ))$, then the $L$-th level of the Lie bracket is expressed as
\begin{equation}\label{eq:chen_tensor}
([\mathbf{x},\mathbf{y}])_{L} = \sum\limits_{\ell=1}^{L-1} [\mathbf{x}_\ell, \mathbf{y}_{L-\ell}] = \sum\limits_{\ell=1}^{L-1} (\mathbf{x}_\ell \otimes  \mathbf{y}_{L-\ell} -  \mathbf{x}_{L-\ell} \otimes \mathbf{y}_\ell),
\end{equation}
analogously to Chen's identity \eqref{eq:chen}.
Then we have the following corollary of \Cref{thm:aBCH}.
\begin{corollary}\label{cor:abch_ambient}
The elements $\mathbf{b}_L\in (\mathbb{R}^{d})^{\otimes L} $   of the tensor series of  $\mathbf{b} = \log (\mathbf{m}^{-1})$, compare \eqref{eq:tensor_coefficient_barycenter},
can be computed as
\[
\mathbf{b}_L = -\sum\limits_{i=1}^N w_i (\mathbf{c}^{(i)}_L +R_L(\mathbf{b},\mathbf{c}^{(i)})),
\]
where $R_L$ is a tensor product analogue of the polynomial defined in \Cref{eq:final_mj_BCHcancl}, i.e.,
\begin{equation}\label{eq:abch_ambient}
R_L (\cdot) = \Big(\sum\limits_{\substack{k=1,\dots,L \\ k \text{ odd}}}\aBCH(\mathbf{b},\mathbf{c}^{(i)})\Big)_L.
\end{equation}
\end{corollary}
\begin{example}
For $L$, we have $R_L = 0$ in \eqref{eq:abch_ambient}, which agrees with 
\Crefrange{eq:m_ambient_L1}{eq:m_ambient_L2}.
For $L=3$, we can apply \eqref{eq:chen_tensor} to $([[\mathbf{b},\mathbf{c}],\mathbf{c}])_3$ to get
\[
R_3(\mathbf{b},\mathbf{c}) =  \frac{1}{12} 
[[\mathbf{b}_1,\mathbf{c}_1],\mathbf{c}_1] = 
 \frac{1}{12}
 (\mathbf{c}_1 \otimes \mathbf{c}_1 \otimes \mathbf{b}_1+  \mathbf{b}_1 \otimes \mathbf{c}_1 \otimes \mathbf{c}_1 - 2 \mathbf{c}_1 \otimes \mathbf{b}_1\otimes \mathbf{c}_1 ).
\]
For order $4$, similarly, we get
\begin{align}
&R_4(\mathbf{b},\mathbf{c}) 
 = \Big(\frac{1}{12}[[\mathbf{b},\mathbf{c}],\mathbf{c}]\Big)_4 =  \frac{1}{12} 
(\underbrace{[[\mathbf{b}_2,\mathbf{c}_1],\mathbf{c}_1]}_{I} + \underbrace{[[\mathbf{b}_1,\mathbf{c}_2],\mathbf{c}_1]+ [[\mathbf{b}_1,\mathbf{c}_1],\mathbf{c}_2]}_{II}) \\
& = \frac{1}{12}
 (\mathbf{c}_1   \mathbf{c}_1 \mathbf{b}_2+  \mathbf{b}_2  \mathbf{c}_1\mathbf{c}_1 - 2\, \mathbf{c}_1  \mathbf{b}_2 \mathbf{c}_1 ) \label{eq:R4_term1}\tag{I}\\
& + \frac{1}{12} (\mathbf{c}_1   \mathbf{c}_2  \mathbf{b}_1+  \mathbf{b}_1  \mathbf{c}_1  \mathbf{c}_2 + \mathbf{c}_2  \mathbf{c}_1  \mathbf{b}_1+  \mathbf{b}_1  \mathbf{c}_2  \mathbf{c}_1 - 2\,\mathbf{c}_2  \mathbf{b}_1 \mathbf{c}_1  -2\, \mathbf{c}_1  \mathbf{b}_1 \mathbf{c}_2).\label{eq:R4_term2}\tag{II}
\end{align}
where in \eqref{eq:R4_term1} and \eqref{eq:R4_term2} we omitted the tensor products for short.
\end{example}
Note that $R_3$ and $R_4$ have $3$ and $9$ terms (``monomials'') respectively, and $R_5$ (provided in Supplementary Materials) has $43$ terms. 
This leads us to the following conjecture.
\begin{conjecture}
For $d \ge L$, the maximal possible number of terms $Q_{L,d} = Q_{L,L}$ in \Cref{table:r_j} coincides with the number of terms obtained from the expansion in \eqref{eq:abch_ambient}.
\end{conjecture} 
\begin{remark}
Note that $Q_{L,L}$ is lower bounded  by the number of terms in \eqref{eq:abch_ambient} if the Lyndon basis is used for $\mathcal{B}$, which follows from  \cite[Theorem 5.3]{bib:R1993} applied to $[\w{1},[\w{2},[ \ldots, \w{L}]\ldots ]$. 
\end{remark}

\section{Linear projections and barycenter of the Brownian motion}\label{sec:pi1}
In this section, we use the fact that the logarithm can be represented by a linear projection in the ambient space.
\subsection{Coordinates of the first kind and barycenter}
Let $T(( \R^ d ))$ be the space 
of formal tensor series. 
It is known, e.g. \cite[section 3.2]{bib:R1993},
that there exists a (unique) linear operator $\pi_1: T((\R^d)) \to T((\R^d))$ that coincides with the logarithm on all group-like elements:
we have 
\begin{equation*}
\pi_1 (\mathbf{x}) = \log(\mathbf{x}), \quad \text{for all } \mathbf{x} \in \G
.\end{equation*}
The operator $\pi_1$ is graded, i.e., there exist linear operators $\pi_{1,s}: (\R^d)^{\otimes s} \to (\R^d)^{\otimes s}$ such that 
\begin{equation*}
\pi_1((1, \mathbf{x}_1,  \mathbf{x}_2, \ldots,  \mathbf{x}_L, \ldots)) = (0, \pi_{1,1}( \mathbf{x}_1), \pi_{1,2}( \mathbf{x}_2), \ldots, \pi_{1,L}( \mathbf{x}_L),\ldots),
\end{equation*}
the $s$-th level of $\pi_1(\mathbf{x})$ is given by
\begin{equation*}
(\pi_1 (\mathbf{x}))_s = (\log(\mathbf{x}))_s = \pi_{1,s}(\mathbf{x}_s) .
\end{equation*}
\begin{example}
The first two levels are given by
\begin{equation}\label{eq:pi1_12}
\pi_{1,1} (\mathbf{x}_1) = \mathbf{x}_1,\quad \pi_{1,2} (\mathbf{x}_2) = \frac{1}{2}(\mathbf{x}_2 - \mathbf{x}_2^{\top}),
\end{equation}
where ${(\cdot)^\top}$ denotes the transposition of $x_2 \in (\R^d)^{\otimes 2}$ viewed as a matrix.
\end{example}

The linearity of $\pi_1$ helps to easily verify the integrability of the expressions appearing in the definition of the barycenter (see \Cref{eq:barycenterCondition}).
\begin{lemma}\label{lem:integral_existence}
Let $\nu$ be the probability measure on 
${G} = \mathcal{G}_{\le L}(\R^d)$ such that the integral $\int_G  \mathbf{x}  \,\nu(\mathrm d\mathbf{x})$
is finite. Then for any $\mathbf{a} \in\mathcal{G}_{\le L}$, the integrals 
\[
  \int_G \log( \mathbf{a} \mathbf{x} ) \,\nu(\mathrm d\mathbf{x})
\]
are well-defined and finite.
\end{lemma}
\begin{proof}
Thanks to the properties of $\pi_1$, we have
\begin{equation}\label{eq:barycenter_via_pi1}
\int_G \log( \mathbf{a} \mathbf{x} ) \,\nu(\mathrm d\mathbf{x}) =
\int_G \pi_1( \mathbf{a} \mathbf{x} ) \,\nu(\mathrm d\mathbf{x}).
\end{equation}
We note that the integrand on the right side is a linear map $  T_{\le L}(\R^d) \to T_{\le L}(\R^d)$, due to linearity of $\pi_1$, linearity of the group operation $\mathbf{a} \mathbf{x}$ and due to Chen's identity \eqref{eq:chen}.
\end{proof}
We also make a side remark that
\Cref{lem:integral_existence}  gives an alternative to \Cref{algo:group-mean-ambient} (though not very practical), computing the barycenter in the ambient space.

\begin{corollary}\label{cor:barycenter_pi1}
For the discrete measure (as in \Cref{prop:1}) and 
\[
\mathbf{a} := \mathbf{m}^{-1}  =  (1,\mathbf{a}_1,\mathbf{a}_2, \ldots,\mathbf{a}_L),
\]
 the $s$-th graded component of the logarithm of $\mathbf{m}$ is a linear combination of the
expected log signature at order $s$ and orders smaller then $s$, 
\[
\log (\mathbf{m})_s = -\log (\mathbf{a})_s=
\sum_{i=1}^{N} w^{(i)} \log ( \mathbf{x}^{(i)})_{L} +
\pi_{1,s} \left(\sum\limits_{k=1}^{s-1}   \mathbf{a}_{s-k} \otimes\left(\sum_{i=1}^{N}  w^{(i)} \mathbf{x}^{(i)}_{k} \right) \right).
\]
\end{corollary}
\begin{proof}
Follows by rewriting the barycenter equation using \eqref{eq:barycenter_via_pi1}  and  \eqref{eq:chen} at order $s$ as
\[
0 = \sum_{i=1}^{N} w^{(i)} \pi_{1,s} \left((\mathbf{a}  \mathbf{x}^{(i)})_s \right) = 
\mathbf{a}_s + \sum_{i=1}^{N} w^{(i)} \left(\pi_{1,s} (\mathbf{x}^{(i)})_s 
+  \sum\limits_{k=1}^{s-1} \pi_{1,s}\left(\mathbf{a}_{s-k} \otimes  \mathbf{x}^{(i)}_{k}\right)\right). 
\]
\end{proof}

\begin{example}
From \Cref{cor:barycenter_pi1} and using the fact that   $\pi_{1,2}$ is an antisymmetrization \eqref{eq:pi1_12}, we can get an alternative proof for the log-signature of the barycenter at order $2$:
\begin{align*}
\log (\mathbf{m})_2 & = \sum_{i=1}^{N} w^{(i)}{\pi_{1,2}(\mathbf{x}^{(i)}_{2})} + 
\pi_{1,2}\Big(a_1 \otimes \underbrace{\sum_{i=1}^{N} w^{(i)}\mathbf{x}^{(i)}_{1}}_{=-\mathbf{a}_1} \Big) = \sum_{i=1}^{N}w^{(i)}(\log \mathbf{x}^{(i)})_2
\end{align*}
\end{example}

\subsection{Explicit formula for projection}
Although $\pi_1$  is classically expressed  in terms of shuffle products \cite[section 3.2]{bib:R1993}, in what follows we need an explicit combinatorial formula for $\pi_1$  that is based on permutations of tensor dimensions that generalizes \eqref{eq:pi1_12}.
\begin{definition}
For a permutation  $\sigma \in S_L$,  the operator\footnote{$\mathsf{permute}_{\sigma}$ coincides with the ``permute'' function in many programming environments (Python, MATLAB).} $\mathsf{permute}_\sigma: (\R^d)^{\otimes L} \to (\R^d)^{\otimes L}$ is 
\begin{equation}\label{eq:permute_definition}
(\mathsf{permute}_{\sigma} \mathbf{a})_{i_{\sigma(1)}, \ldots, i_{\sigma(L)}} = {\mathbf{a}}_{i_1,\ldots,i_L},
\end{equation}
that is, the $k$-th dimension of the tensor $\mathbf{a}$ is put into $\sigma(k)$-th position.
\end{definition}

\begin{definition}
For a permutation  $\sigma \in S_L$,   $\desc(\sigma)$   is an integer between $0$ and $L-1$, equal to the number of the descents of the permutation, e.g., $\desc((3142)) = 2$.
\end{definition}
We use the following formula for $\pi_{1,L}$ based on \cite[Lemma 3.9]{bib:R1993}. We did not find this explicit form elsewhere, this is why we provide a short proof.
\begin{proposition}\label{prop:pi1_permutations}
\begin{equation}\label{eq:pi1_permutations}
\pi_{1,L} (\mathbf{a}) = \sum\limits_{\sigma \in S_L} \frac{(-1)^{\desc(\sigma)}}{L}
{\binom{L-1}{\desc(\sigma)}}^{-1} \mathsf{permute}_{\sigma}(\mathbf{a}).
\end{equation}
\end{proposition}
Before giving the proof of \Cref{prop:pi1_permutations}, we are going to provide some examples.
We will also use a shorthand for the reversal of the dimensions:
\begin{equation*}
\rev(\mathbf{a}) = \mathsf{permute}_{(L,L-1,\ldots,1)} (\mathbf{a}),
\end{equation*}
i.e., $(\rev(\mathbf{a}))_{i_1,\ldots, i_L} = {\mathbf{a}}_{i_L,\ldots,i_1}$.
We observe that \eqref{eq:pi1_permutations} agrees with \eqref{eq:pi1_12} at orders $1$ and $2$.
\begin{example}
Indeed, at order $2$ the formula \eqref{eq:pi1_permutations} yields
\begin{equation*}
\pi_{1,2} (\mathbf{a})  = \frac{1}{2} \mathbf{a} -  \frac{1}{2} \rev (\mathbf{a}) = \frac{1}{2}( \mathbf{a}  - \mathbf{a}^{\top}).
\end{equation*}
At order $3$ we get
\begin{align*}
\pi_{1,3} (\mathbf{a}) &= \frac{1}{3} \left(\mathbf{a} + \rev(\mathbf{a})\right) \\
&- 
\frac{1}{6} \left( \mathsf{permute}_{(1,3,2)}( \mathbf{a}) + \mathsf{permute}_{(2,1,3)}( \mathbf{a}) + \mathsf{permute}_{(2,3,1)}( \mathbf{a}) + \mathsf{permute}_{(3,1,2)} (\mathbf{a})\right), 
\end{align*}
that is, in coordinates, 
\begin{equation*}
(\pi_{1,3} (\mathbf{a}))_{i,j,k} = \frac{1}{3} (\underbrace{{\mathbf{a}}_{i,j,k}}_{(123)} + \underbrace{{\mathbf{a}}_{k,j,i}}_{(321)}) - 
\frac{1}{6} (\underbrace{{\mathbf{a}}_{i,k,j}}_{(132)} + \underbrace{{\mathbf{a}}_{j,i,k}}_{(213)} + \underbrace{{\mathbf{a}}_{j,k,i}}_{(312)} +\underbrace{{\mathbf{a}}_{k,i,j}}_{(231)}). 
\end{equation*}
\end{example}

\begin{proof}[Proof of \Cref{prop:pi1_permutations}]
We recall that by \cite[Corollary 3.16]{bib:R1993}, the operator $\pi_{1,L}$ applied to the word $\w{12}\ldots\w{L}$ is equal to
\begin{equation}\label{eq:pi1_permutations_single_word}
\pi_{1,L} (\w{12}\ldots \w{L}) = \sum\limits_{\sigma \in S_L} \frac{(-1)^{\desc(\sigma)}}{L}
{\binom{L-1}{\desc(\sigma)}}^{-1} \sigma,
\end{equation}
where the permutations are viewed as a word.
Next, we show that a similar to \eqref{eq:pi1_permutations_single_word} formula holds for any word.

Let $i_1 \ldots i_L$ be a word $\phi$ the concatenation homomorphism of words (i.e. $\phi(xy) = \phi(x)\phi(y)$ for any two words) defined by $\phi(k) = i_k$.
Such a homomorphism, by \cite[Lemma 3.9]{bib:R1993}, commutes with $\pi_1$, therefore we have for any word $i_1\ldots i_L$, 
\begin{equation*}
\pi_{1,L} (i_1\ldots i_L) = \sum\limits_{\sigma \in S_L} \frac{(-1)^{\desc(\sigma)}}{L}
{\binom{L-1}{\desc(\sigma)}}^{-1} \phi \circ \sigma. 
\end{equation*}
Now, in the tensor space, the word $i_1\ldots i_L$ corresponds to the tensor
\begin{equation}\label{eq:tensor_basis}
\mathbf{b} = e_{i_1} \otimes \cdots  \otimes e_{i_L},
\end{equation}
and it is easy to verify that, by \eqref{eq:permute_definition}, $\phi \circ \sigma$ corresponds to the tensor
\begin{equation*}
\mathsf{permute}_{\sigma}(\mathbf{b}) =  e_{i_{\sigma(1)}} \otimes \cdots  \otimes e_{i_{\sigma(L)}}.
\end{equation*}
Hence, the formula \eqref{eq:pi1_permutations} is valid for any basis tensor \eqref{eq:tensor_basis}, and hence is valid for any $\mathbf{a}$.
\end{proof}

\subsection{Barycenter of the Brownian motion}
The goal of this section is to show, using the previous results on $\pi_1$, that the barycenter of the Brownian motion is the identity.
For background on the Brownian motion and its signature
we refer to \cite{bib:FV2010,bib:Bau2004}.

Using this result, we can show that the expected log-signature vanishes, and as a consequence, that the barycenter is the identity element.
\begin{proposition}\label{prop:barycenter_BM}
Let $\Sigma \in \mathbb{R}^{d\times d}$ be positive definite and   $X_t$ be the Brownian motion on $[0,1]$ with covariance $\Sigma$ defined by $X_t = \sqrt{\Sigma} B_t$, with standard Brownian motion $B_t$.
Then we have
\begin{equation*}
\mathbf{E}(\log{\IIS(X_t)}) = 0,
\end{equation*}
which implies with \cref{lem:barycenter_naive} that 
$\mathbf{m}_{\IIS(X_t)} = \e$.
\end{proposition}

Before proving \cref{prop:barycenter_BM}, we make a remark about the symmetries in the problem.
\begin{lemma}\label{lem:reversion_symmetry}
Let $\mathbf{a} \in (\R^{d})^{L}$, for $L$ even, that is with respect to reversion, 
\begin{equation}\label{eq:reversion_symmetry}
\mathbf{a} = \rev(\mathbf{a}).
\end{equation}
Then we have that 
$
\pi_1 (\mathbf{a})  = 0
$.
\end{lemma}
\begin{proof}
We start by finding a formula for $\pi_1(\rev(\mathbf{a}))$.
It is easy to see that
\[
\mathsf{permute}_{\sigma}(\rev(\mathbf{a}))  = \mathsf{permute}_{\alpha \circ \sigma}(\mathbf{a}),
\]
where $\alpha$ is the substitution
\[
\alpha =\begin{pmatrix}
1 & \cdots & L \\
\downarrow & & \downarrow \\
L & \cdots & 1
\end{pmatrix},
\]
i.e., $\alpha((2314)) = (3241)$.
Hence,  after reparameterization, and using $\alpha^{-1} = \alpha$, we get
\begin{align*}
\pi_{1,L} (\rev(\mathbf{a})) &= \sum\limits_{\sigma \in S_L} \frac{(-1)^{\desc(\sigma)}}{L}
{\binom{L-1}{\desc(\sigma)}}^{-1} \mathsf{permute}_{\alpha \circ \sigma}(\mathbf{a})\\&=  \sum\limits_{\sigma \in S_L} \frac{(-1)^{\desc(\alpha \circ \sigma)}}{L}
{\binom{L-1}{\desc(\alpha \circ \sigma)}}^{-1} \mathsf{permute}_{\sigma}(\mathbf{a}).
\end{align*}
Next, note that the number of descents of $\alpha \circ \sigma$ is the number of rises in $\sigma$, or equivalently,
\[
\desc(\sigma) + \desc(\alpha \circ \sigma)= L-1, 
\]
which implies 
\[
\binom{L-1}{\desc(\alpha \circ \sigma)} =\binom{L-1}{\desc(\sigma)}
\] and $(-1)^{\desc(\alpha \circ \sigma)}  = (-1)^{L-1}(-1)^{\desc(\sigma)}$. 
Together with the symmetry of $\mathbf{a}$, we get
\[
2 \pi_1(\mathbf{a}) = \pi_1(\mathbf{a}) + \pi_1(\rev(\mathbf{a}))
=
 \sum\limits_{\sigma \in S_L} \frac{(-1)^{\desc(\sigma)} (1 + (-1)^{L})}{L}
{\binom{L-1}{\desc(\sigma)}}^{-1} \mathsf{permute}_{\sigma}(\mathbf{a}),
\]
which is equal to $0$ if $L$ is even.
\end{proof}

\begin{proof}[Proof of \Cref{prop:barycenter_BM}]
By a well-known result on the expected signature of Brownian motion (compare  {\cite[\S 7.1]{bib:AFS2019}, \cite[Theorem 3.9]{bib:FH2020}}), we have
\[
\mathbf{E}({\IIS(X_t)}) = \exp\left(\frac{1}{2} \Sigma\right) = \frac{1}{2}\left(1, 0, \Sigma,0, \frac{1}{2} \Sigma \otimes \Sigma, 0, \frac{1}{3!} \Sigma \otimes \Sigma\otimes \Sigma, \ldots\right).
\] 
Then, by linearity,
\[
\mathbf{E}(\log{\IIS(X_t)}) = \mathbf{E}(\pi_1{\IIS(X_t)}) = \pi_1 (\mathbf{E}({\IIS(X_t)})).
\]
We see that at odd orders we have $(\mathbf{E}({\IIS(X_t)}))_L = 0$, hence  
$\pi_{1,L}((\mathbf{E}({\IIS(X_t)}))_L) = 0$.
At even order, the elements $(\mathbf{E}({\IIS(X_t)}))_L$ clearly satisfy the reversal symmetry \eqref{eq:reversion_symmetry}, hence $\pi_{1,L}((\mathbf{E}({\IIS(X_t)}))_L) = 0$.
\end{proof}

\section{Open questions / Outlook}

\begin{itemize}

    \item
    Time series data is usually discrete-in-time,
    and the recently introduced iterated-sums signature (or discrete signature)
    provides a natural way to deal with such data, without the need to interpolate \cite{bib:DEFT2020}.
    For an appropriately truncated version,
    one is again in the setting of
    a free nilpotent group.
    What is different in that setting
    is that not all group elements
    can be realized as the signature of
    a time series \cite[p.279]{bib:DEFT2020}.
    Can all barycenters be realized, though?

    \item
    The group mean stores less information about the measure than
    the expected signature in the ambient space.
    Indeed, there are many measures
    on the signature space that have a group mean
    equal to the identity, e.g., 
 for every Brownian motion
    with arbitrary, time-homogeneous variance-structure, \Cref{prop:barycenter_BM}.
    Hence, the group mean does not determine the measure.
    The expected signature on the other hand \emph{does}
    determine the measure (under some integrability conditions, which are fulfilled for Brownian motion).
    Can we enrich the group mean with more information to restore
    the identifiability of the measure?
    The concept of ``elicitation'' from statistics might be useful here
    \cite[p.6]{bib:BO2021}.

    \item
    Are there other ways to define a bi-invariant group mean in the signature space?
    The essential property for bi-invariance was conjugation-equivariance of the logarithm,
    \begin{align*}
      \log(\mathbf g^{-1} \mathbf x \mathbf g) = \mathbf g^{-1} \log(\mathbf x) \mathbf g.
    \end{align*}
    Is this the only such map from grouplike elements to the Lie algebra,
    that is invertible?

\end{itemize}

\section*{Acknowledgments}
We would like to thank Harald Oberhauser
for pointing out the (formal) connection
to proper scoring rules (\Cref{rem:groupmean}).
We also thank Bernd Sturmfels for  suggesting to consider Gr\"{o}bner bases of modules.

\bibliographystyle{alpha}
\bibliography{references}

\appendix

\section{Tables}

\Cref{table:B_Ld} illustrates the dimension of the truncated Lie algebra, see Equation~\eqref{eq:dim}. \Cref{table:B_Ld_2} shows the dimension of the ambient tensor algebra. \Cref{table:p_j} shows the maximal number of summands in $p_j$. This last table should be put in perspective with \Cref{table:r_j} from the main document.

\begin{table}
\begin{center}
\begin{small}
    \begin{tabular}{ c| c c c c c c }
  $B_{L,d}$ & $d=2$ & $d=3$ & $d=4$ & $d=5$ & $d=6$ & $d=7$ \\\hline
  $L=2$ & $3$ & $6$ & $10$ & $15$ & $21$ &$28$ \\
  $L=3$ & $5$ & $14$& $30$& $55$ & $91$ &$140$ \\
  $L=4$ & $8$ & $32$ & $90$ & $205$ & $406$ & $728$ \\
  $L=5$ & $14$ &$80$ & $294$ & $829$ & $1960$ & $4088$ 
\end{tabular}
\end{small}
\caption{Dimension $B_{L,d}=\dim_\R(\LieAlg_{\leq L})$ due to \Cref{eq:dim} for various $d$ and $L$.}
\label{table:B_Ld}
\end{center}
\end{table}
\begin{table}
\begin{center}
\begin{small}
    \begin{tabular}{ c| c c c c c c }
   & $d=2$ & $d=3$ & $d=4$ & $d=5$ & $d=6$ & $d=7$ \\\hline
  $L=2$ & $7$ & $13$ & $21$ & $31$ & $43$ &$57$ \\
  $L=3$ & $15$ & $40$& $85$& $156$ & $259$ &$400$ \\
  $L=4$ & $31$ & $121$ & $341$ & $781$ & $1555$ & $2801$ \\
  $L=5$ & $63$ &$364$ & $1365$ & $3906$ & $9331$ & $19608$ 
\end{tabular}
\end{small}
\caption{The number of entries \(\sum_{i=0}^{L}d^{i}\)  in an element of \(\G_{\leq L}(\mathbb{R}^d)\) for various $d$ and $L$.}
\label{table:B_Ld_2}
\end{center}
\end{table}


\begin{table}
    \begin{center}
    \begin{small}
    \begin{tabular}{ c| c c c c c c c}
   & $d=2$ & $d=3$ & $d=4$ & $d=5$ & $d=6 $&$\cdots$& $d=10$\\\hline
  $L=2 $& $2 $& $2 $& $2 $& $2 $& $2 $&$\cdots$& $2$\\
  $L=3 $& $6 $& $10 $&$ 10 $& $10 $& 10 &$\cdots$& $10$\\
  $L=4 $& $12 $& $24 $& $30$& $30 $& $30$&$\cdots$& $30$\\
  $L=5 $& $32 $& $64 $& $84 $& $98$& $98$&$\cdots$& $98$
\end{tabular}
\end{small}
\caption{Maximal number of sumands in $p_j$ for various $d$ and $L$.}
\label{table:p_j}
\end{center}
\end{table}

\section{Algorithms}
Similarly as the normal form algorithm (\Cref{algo:normalForm} from the main document) we can formulate Buchberger's algorithm  with frozen $C$. 

\LinesNotNumbered
\begin{algorithm2e}
\DontPrintSemicolon 
\KwIn{A set $F\subseteq R$ of polynomials, and a monomial ordering $<$ on $R$}
\KwOut{A set $\operatorname{Buchberger}(F)=G\subseteq R$ of polynomials}
\setcounter{AlgoLine}{0}
\nl $G \gets F$\;
\nl \While{\(\exists g,h\in G\,\exists w_i,v_j:M_{w_1}\dots M_{w_\ell}\operatorname{lm}(g)=M_{v_1}\dots M_{v_t}\operatorname{lm}(h)\)}{
    \nl $r\gets \operatorname{rNF}(\frac{1}{\operatorname{lc}(g)}M_{w_1}\dots M_{w_\ell}g- \frac{1}{\operatorname{lc}(h)}M_{v_1}\dots M_{v_t}h,G)$\;
    \nl \If{$r\not=0$}{\nl $G\gets G\cup\{r\}$}
    }
\nl \Return{ $G$}\;
\caption{\textsc{Buchberger's algorithm (with frozen $C$)}}
\label{algo:Buchberger}
\end{algorithm2e}

The following lemma verifies the first two rows of \Cref{table:p_j} without the help of computer algebra machinery. 
Furthermore, it illustrates the basic idea of \Cref{thm:complexity} for small truncation levels. 

\begin{lemma}\label{lem:boundQ}
Let $p_j$ be due to \Cref{lem:pol_ex}. 
\begin{enumerate}
  
\item\label{lem:boundSum_Leq2}For all $d\geq L=2$, the number of summands in $p_j$ is bounded by $2$. 
\item\label{lem:boundSum_Leq3}For all $d\geq L=3$, the number of summands in $p_j$ is bounded by $10$. 
\end{enumerate}
\end{lemma}
\begin{proof}
For part \ref{lem:boundSum_Leq2} 
we observe that the bound is reached since \Cref{symbolic_bch} reduces to 
\begin{align}\label{eq:p2}
\mathcal{B}_{\w{ij}}^*(\BCH(X,Y))&=
\mathcal{B}_{\w{ij}}^*((C_{\w{ij}}+M_{\w{ij}})\mathcal{B}_{\w{ij}}
+\frac{1}{2}[M_{\w{i}}\mathcal{B}_{\w{i}},C_{\w{j}}\mathcal{B}_{\w{j}}]
+\frac{1}{2}[M_{\w{j}}\mathcal{B}_{\w{j}},C_{\w{i}}\mathcal{B}_{\w{i}}])\nonumber\\
&=M_{\w{ij}}+C_{\w{ij}}+\underbrace{\frac{1}{2}(M_{\w{i}}C_{\w{j}}-M_{\w{j}}C_{\w{i}})}_{=p_{\w{ij}}}
\end{align}
for all $1\leq i<j\leq d$.
Furthermore, it is maximal since for all $i$, 
\begin{equation}\label{eq:p1}\mathcal{B}_{\w{i}}^*(\BCH(X,Y))=\mathcal{B}_{\w{i}}^*((C_{\w{i}}+M_{\w{i}})\mathcal{B}_{\w{i}}=C_{\w{i}}+M_{\w{i}}.
\end{equation}
For part \ref{lem:boundSum_Leq3} with $i<j<k$ we have for instance 
$$\mathcal{B}_{\w{ijk}}^*([\mathcal{B}_{\w{jk}},
\mathcal{B}_{\w{i}}])=\mathcal{B}_{\w{ijk}}^*(-[\mathcal{B}_{\w{i}},\mathcal{B}_{\w{jk}}])=
-\mathcal{B}_{\w{ijk}}^*(\mathcal{B}_{\w{ijk}})=-1$$
or via the Jacobi identity, 
$$\mathcal{B}_{\w{ijk}}^*([\mathcal{B}_{\w{k}},[\mathcal{B}_{\w{i}},
\mathcal{B}_{\w{j}}]])
=\mathcal{B}_{\w{ijk}}^*(-[\mathcal{B}_{\w{i}},[\mathcal{B}_{\w{j}},
\mathcal{B}_{\w{k}}]]
-[\mathcal{B}_{\w{j}},[\mathcal{B}_{\w{k}},
\mathcal{B}_{\w{i}}]])=
\mathcal{B}_{\w{ijk}}^*(-\mathcal{B}_{\w{ijk}}-\mathcal{B}_{\w{ikj}})=-1,$$
and thus, we obtain 
\begin{align*}
\mathcal{B}_{\w{ijk}}^*(\BCH(X,Y))&=
\mathcal{B}_{\w{ijk}}^*((C_{\w{ijk}}+M_{\w{ijk}})\mathcal{B}_{\w{ijk}}\\
&\;\;\;\;\;\;\;\;\;\;
+\frac{1}{2}(
[M_{\w{ij}}\mathcal{B}_{\w{ij}},
C_{\w{k}}\mathcal{B}_{\w{k}}]
+[M_{\w{i}}\mathcal{B}_{\w{i}},
C_{\w{jk}}\mathcal{B}_{\w{jk}}]\\
&\;\;\;\;\;\;\;\;\;\;\;\;\;\;\;\;\;\;\;\;\;\;
+
[M_{\w{jk}}\mathcal{B}_{\w{jk}},
C_{\w{i}}\mathcal{B}_{\w{i}}]
+[M_{\w{k}}\mathcal{B}_{\w{k}},
C_{\w{ij}}\mathcal{B}_{\w{ij}}]
)\\
&\;\;\;\;\;\;\;\;\;\;
+ \frac{1}{12}(
[M_{\w{i}}\mathcal{B}_{\w{i}},[M_{\w{j}}\mathcal{B}_{\w{j}},
C_{\w{k}}\mathcal{B}_{\w{k}}]]
+[M_{\w{i}}\mathcal{B}_{\w{i}},[M_{\w{k}}\mathcal{B}_{\w{k}},
C_{\w{j}}\mathcal{B}_{\w{j}}]]
\\
&\;\;\;\;\;\;\;\;\;\;\;\;\;\;\;\;\;\;\;\;\;\;
+ [M_{\w{k}}\mathcal{B}_{\w{k}},[M_{\w{i}}\mathcal{B}_{\w{i}},
C_{\w{j}}\mathcal{B}_{\w{j}}]]
+[M_{\w{k}}\mathcal{B}_{\w{k}},[M_{\w{j}}\mathcal{B}_{\w{j}},
C_{\w{i}}\mathcal{B}_{\w{i}}]])\\
&\;\;\;\;\;\;\;\;\;\;
- \frac{1}{12}(
[C_{\w{i}}\mathcal{B}_{\w{i}},[M_{\w{j}}\mathcal{B}_{\w{j}},
C_{\w{k}}\mathcal{B}_{\w{k}}]]
+[C_{\w{i}}\mathcal{B}_{\w{i}},[M_{\w{k}}\mathcal{B}_{\w{k}},
C_{\w{j}}\mathcal{B}_{\w{j}}]]
\\
&\;\;\;\;\;\;\;\;\;\;\;\;\;\;\;\;\;\;\;\;\;\;
+ [C_{\w{k}}\mathcal{B}_{\w{k}},[M_{\w{i}}\mathcal{B}_{\w{i}},
C_{\w{j}}\mathcal{B}_{\w{j}}]]
+[C_{\w{k}}\mathcal{B}_{\w{k}},[M_{\w{j}}\mathcal{B}_{\w{j}},
C_{\w{i}}\mathcal{B}_{\w{i}}]]))\\
&=M_{\w{ijk}}+C_{\w{ijk}}+ \frac{1}{2}(M_{\w{ij}}C_{\w{k}}
+M_{\w{i}}C_{\w{jk}}
-M_{\w{jk}}C_{\w{i}}
-M_{\w{k}}C_{\w{ij}})\\
&\;\;\;\;\;\;\;\;\;\;
+ \frac{1}{12}(
M_{\w{i}}M_{\w{j}}C_{\w{k}}
-2M_{\w{i}}M_{\w{k}}C_{\w{j}}
+M_{\w{k}}M_{\w{j}}C_{\w{i}})\\
&\;\;\;\;\;\;\;\;\;\;
+ \frac{1}{12}(
-2C_{\w{i}}M_{\w{j}}C_{\w{k}}
+C_{\w{i}}M_{\w{k}}C_{\w{j}}
+C_{\w{k}}M_{\w{i}}C_{\w{j}}). 
\end{align*}
Similarly one can show that  $\w{iij}$, $\w{ijj}$ and $\w{ikj}$  yield polynomials $p_j$ with number of sumands bouned by $10$. 
\end{proof}

\section{Illustrative Gr\"{o}bner reduction with order 4 and dimension 3}\label{ex:Spolys_needed}

The following example illustrates  \Cref{cor:rj_def}. In paricular it shows that Buchberger's algorithm can reduce the number of summands in the polynomials derived from the Baker--Campbell--Hausdorff formula. \\\\
Let $d=3$ and $L=4$, thus $B=32$.
We choose  the degree lexicographic order with 
    $$C_1>C_2>\cdots>C_B>M_1>M_2>\cdots>M_B.$$
    
    We compute $q$ due to \Cref{lemma:init_q}. This is analogous to \Cref{lem:boundQ}, e.g., the first $19$ relations are 
    \begin{small}
    \begin{align*}
        &q_1=C_{1} - M_{1}\\&
q_{2}= C_{2} - M_{2}\\&
q_{3}= C_{3} - M_{3}\\&
q_{4}= -\frac{1}{2}C_{2}M_{1} + \frac{1}{2}C_{1}M_{2} + C_{4} - M_{4}\\&
q_{5}= -\frac{1}{2}C_{3}M_{1} + \frac{1}{2}C_{1}M_{3} + C_{5} - M_{5}\\&
q_{6}= -\frac{1}{2}C_{3}M_{2} + \frac{1}{2}C_{2}M_{3} + C_{6} - M_{6}\\&
q_{7}= \frac{1}{12}C_{1}C_{2}M_{1} + \frac{1}{12}C_{2}M_{1}^2 - \frac{1}{12}C_{1}^2M_{2} - \frac{1}{12}C_{1}M_{1}M_{2} - \frac{1}{2}C_{4}M_{1} + \frac{1}{2}C_{1}M_{4} + C_{7} - M_{7}\\&
q_{8}= \frac{1}{12}C_{1}C_{3}M_{1} + \frac{1}{12}C_{3}M_{1}^2 - \frac{1}{12}C_{1}^2M_{3} - \frac{1}{12}C_{1}M_{1}M_{3} - \frac{1}{2}C_{5}M_{1} + \frac{1}{2}C_{1}M_{5} + C_{8} - M_{8}\\&
q_{9}= -\frac{1}{12}C_{2}^2M_{1} + \frac{1}{12}C_{1}C_{2}M_{2} - \frac{1}{12}C_{2}M_{1}M_{2} + \frac{1}{12}C_{1}M_{2}^2 + \frac{1}{2}C_{4}M_{2} - \frac{1}{2}C_{2}M_{4} + C_{9} - M_{9}\\&
q_{10}= -\frac{1}{12}C_{2}C_{3}M_{1} + \frac{1}{6}C_{1}C_{3}M_{2} + \frac{1}{12}C_{3}M_{1}M_{2} - \frac{1}{12}C_{1}C_{2}M_{3} - \frac{1}{6}C_{2}M_{1}M_{3} + \frac{1}{12}C_{1}M_{2}M_{3}  \\&\;\;\;\;\;\;\;\;\;\;-\frac{1}{2}C_{6}M_{1} + \frac{1}{2}C_{4}M_{3} - \frac{1}{2}C_{3}M_{4} + \frac{1}{2}C_{1}M_{6} + C_{10} - M_{10}\\&
q_{11}= -\frac{1}{6}C_{2}C_{3}M_{1} + \frac{1}{12}C_{1}C_{3}M_{2} - \frac{1}{12}C_{3}M_{1}M_{2} + \frac{1}{12}C_{1}C_{2}M_{3} - \frac{1}{12}C_{2}M_{1}M_{3} + \frac{1}{6}C_{1}M_{2}M_{3} \\&\;\;\;\;\;\;\;\;\;\;+ \frac{1}{2}C_{5}M_{2} + \frac{1}{2}C_{4}M_{3} - \frac{1}{2}C_{3}M_{4} - \frac{1}{2}C_{2}M_{5} + C_{11} - M_{11}\\&
q_{12}= -\frac{1}{12}C_{3}^2M_{1} + \frac{1}{12}C_{1}C_{3}M_{3} - \frac{1}{12}C_{3}M_{1}M_{3} + \frac{1}{12}C_{1}M_{3}^2 + \frac{1}{2}C_{5}M_{3} - \frac{1}{2}C_{3}M_{5} + C_{12} - M_{12}\\&
q_{13}= \frac{1}{12}C_{2}C_{3}M_{2} + \frac{1}{12}C_{3}M_{2}^2 - \frac{1}{12}C_{2}^2M_{3} - \frac{1}{12}C_{2}M_{2}M_{3} - \frac{1}{2}C_{6}M_{2} + \frac{1}{2}C_{2}M_{6} + C_{13} - M_{13}\\&
q_{14}= -\frac{1}{12}C_{3}^2M_{2} + \frac{1}{12}C_{2}C_{3}M_{3} - \frac{1}{12}C_{3}M_{2}M_{3} + \frac{1}{12}C_{2}M_{3}^2 + \frac{1}{2}C_{6}M_{3} - \frac{1}{2}C_{3}M_{6} + C_{14} - M_{14}\\&
q_{15}= -\frac{1}{24}C_{1}C_{2}M_{1}^2 + \frac{1}{24}C_{1}^2M_{1}M_{2} + \frac{1}{12}C_{1}C_{4}M_{1} + \frac{1}{12}C_{4}M_{1}^2 - \frac{1}{12}C_{1}^2M_{4} - \frac{1}{12}C_{1}M_{1}M_{4}\\&\;\;\;\;\;\;\;\;\;\; - \frac{1}{2}C_{7}M_{1} + \frac{1}{2}C_{1}M_{7} + C_{15} - M_{15}\\&
q_{16}= -\frac{1}{24}C_{1}C_{3}M_{1}^2 + \frac{1}{24}C_{1}^2M_{1}M_{3} + \frac{1}{12}C_{1}C_{5}M_{1} + \frac{1}{12}C_{5}M_{1}^2 - \frac{1}{12}C_{1}^2M_{5} - \frac{1}{12}C_{1}M_{1}M_{5} \\&\;\;\;\;\;\;\;\;\;\;- \frac{1}{2}C_{8}M_{1} + \frac{1}{2}C_{1}M_{8} + C_{16} - M_{16}\\&
q_{17}= \frac{1}{24}C_{2}^2M_{1}^2 - \frac{1}{24}C_{1}^2M_{2}^2 - \frac{1}{12}C_{2}C_{4}M_{1} - \frac{1}{12}C_{1}C_{4}M_{2} - \frac{1}{6}C_{4}M_{1}M_{2} + \frac{1}{6}C_{1}C_{2}M_{4} \\&\;\;\;\;\;\;\;\;\;\;+ \frac{1}{12}C_{2}M_{1}M_{4} + \frac{1}{12}C_{1}M_{2}M_{4} - \frac{1}{2}C_{9}M_{1} + \frac{1}{2}C_{7}M_{2} - \frac{1}{2}C_{2}M_{7} + \frac{1}{2}C_{1}M_{9} + C_{17} - M_{17}\\&
q_{18}= \frac{1}{24}C_{2}C_{3}M_{1}^2 - \frac{1}{12}C_{1}C_{3}M_{1}M_{2} + \frac{1}{12}C_{1}C_{2}M_{1}M_{3} - \frac{1}{24}C_{1}^2M_{2}M_{3} - \frac{1}{12}C_{3}C_{4}M_{1} + \frac{1}{12}C_{1}C_{6}M_{1} \\&\;\;\;\;\;\;\;\;\;\;+ \frac{1}{12}C_{6}M_{1}^2 - \frac{1}{12}C_{1}C_{4}M_{3} - \frac{1}{6}C_{4}M_{1}M_{3} + \frac{1}{6}C_{1}C_{3}M_{4} + \frac{1}{12}C_{3}M_{1}M_{4} + \frac{1}{12}C_{1}M_{3}M_{4} \\&\;\;\;\;\;\;\;\;\;\;- \frac{1}{12}C_{1}^2M_{6} - \frac{1}{12}C_{1}M_{1}M_{6} - \frac{1}{2}C_{10}M_{1} + \frac{1}{2}C_{7}M_{3} - \frac{1}{2}C_{3}M_{7} + \frac{1}{2}C_{1}M_{10} + C_{18} - M_{18}\\
&q_{19}=\frac{1}{12}C_{2}C_{3}M_{1}^2 - \frac{1}{12}C_{1}^2M_{2}M_{3} - \frac{1}{12}C_{3}C_{4}M_{1} - \frac{1}{12}C_{2}C_{5}M_{1} - \frac{1}{12}C_{1}C_{5}M_{2} - \frac{1}{6}C_{5}M_{1}M_{2} \\&\;\;\;\;\;\;\;\;\;\;- \frac{1}{12}C_{1}C_{4}M_{3} - \frac{1}{6}C_{4}M_{1}M_{3} + \frac{1}{6}C_{1}C_{3}M_{4} + \frac{1}{12}C_{3}M_{1}M_{4} + \frac{1}{12}C_{1}M_{3}M_{4} + \frac{1}{6}C_{1}C_{2}M_{5} \\&\;\;\;\;\;\;\;\;\;\;+ \frac{1}{12}C_{2}M_{1}M_{5} + \frac{1}{12}C_{1}M_{2}M_{5} - \frac{1}{2}C_{11}M_{1} + \frac{1}{2}C_{8}M_{2} + \frac{1}{2}C_{7}M_{3} - \frac{1}{2}C_{3}M_{7} - \frac{1}{2}C_{2}M_{8} \\&\;\;\;\;\;\;\;\;\;\;+ \frac{1}{2}C_{1}M_{11} + C_{19} - M_{19}.
    \end{align*}
   \end{small} 
We begin with reductions only, so we do not use S-polynomials from Buchberger's algorithm yet.  After the indicated reductions we obtain 
\begin{small}
\begin{align*}
    &
q_1=C_{1} - M_{1}=:h_{1}\\&
q_2= C_{2} - M_{2}=:h_{2}\\&
q_3= C_{3} - M_{3}=:h_{3}\\&
q_4\xrightarrow{q_1,q_2} C_{4} - M_{4}=:h_{4}\\&
q_5\xrightarrow{q_1,q_3} C_{5} - M_{5}=:h_5\\&
q_{6}\xrightarrow{q_2,q_3} C_{6} - M_{6}=:h_6\\&
q_{7}\xrightarrow{h_i,i\leq 6} \frac{1}{12}C_{1}C_{2}M_{1} - \frac{1}{12}C_{1}^2M_{2} + C_{7} - M_{7}=:h_7\\&
q_{8}\xrightarrow{h_i,i\leq 6} \frac{1}{12}C_{1}C_{3}M_{1} - \frac{1}{12}C_{1}^2M_{3} + C_{8} - M_{8}=:h_8\\&
q_{9}\xrightarrow{h_i,i\leq 6} -\frac{1}{12}C_{2}^2M_{1} + \frac{1}{12}C_{1}C_{2}M_{2} + C_{9} - M_{9}=:h_9\\&
q_{10}\xrightarrow{h_i,i\leq 6} -\frac{1}{12}C_{2}C_{3}M_{1} + \frac{1}{6}C_{1}C_{3}M_{2} - \frac{1}{12}C_{1}C_{2}M_{3} + C_{10} - M_{10}=:\tilde h_{10}\\
&\;\;\;\;\;\xrightarrow{h_{11}} -\frac{1}{12}C_{2}C_{3}M_{1} + \frac{1}{12}C_{1}C_{2}M_{3} - \frac{1}{3}C_{10} + \frac{2}{3}C_{11} + \frac{1}{3}M_{10} - \frac{2}{3}M_{11}=:h_{10}\\&
q_{11}\xrightarrow{h_i,i\leq 6} -\frac{1}{6}C_{2}C_{3}M_{1} + \frac{1}{12}C_{1}C_{3}M_{2} + \frac{1}{12}C_{1}C_{2}M_{3} + C_{11} - M_{11}=:\tilde h_{11}\\
&\;\;\;\;\;\xrightarrow{ \tilde h_{10}}-\frac{1}{4}C_{1}C_{3}M_{2} + \frac{1}{4}C_{1}C_{2}M_{3} - 2C_{10} + C_{11} + 2M_{10} - M_{11}=:h_{11}\\&
q_{12}\xrightarrow{h_i,i\leq 6} -\frac{1}{12}C_{3}^2M_{1} + \frac{1}{12}C_{1}C_{3}M_{3} + C_{12} - M_{12}=:h_{12}\\&
q_{13}\xrightarrow{h_i,i\leq 6} \frac{1}{12}C_{2}C_{3}M_{2} - \frac{1}{12}C_{2}^2M_{3} + C_{13} - M_{13}=:h_{13}\\&
q_{14}\xrightarrow{h_i,i\leq 6} -\frac{1}{12}C_{3}^2M_{2} + \frac{1}{12}C_{2}C_{3}M_{3} + C_{14} - M_{14}=:h_{14}\\&
q_{15}\xrightarrow{h_i,i\leq 6} \frac{1}{12}C_{1}C_{4}M_{1} - \frac{1}{12}C_{1}^2M_{4} + C_{15} - M_{15}=:h_{15}\\&
q_{16}\xrightarrow{h_i,i\leq 6} \frac{1}{12}C_{1}C_{5}M_{1} - \frac{1}{12}C_{1}^2M_{5} + C_{16} - M_{16}=:h_{16}\\&
q_{17}\xrightarrow{h_9}\frac{1}{24}C_{1}C_{2}M_{1}M_{2} - \frac{1}{24}C_{1}^2M_{2}^2 - \frac{1}{12}C_{2}C_{4}M_{1} - \frac{1}{12}C_{1}C_{4}M_{2} - \frac{1}{6}C_{4}M_{1}M_{2} + \frac{1}{6}C_{1}C_{2}M_{4} \\&\;\;\;\;\;\;\;\;\;\;\;\;\;+ \frac{1}{12}C_{2}M_{1}M_{4} + \frac{1}{12}C_{1}M_{2}M_{4} + \frac{1}{2}C_{7}M_{2} - \frac{1}{2}C_{2}M_{7} + \frac{1}{2}C_{1}M_{9} - \frac{1}{2}M_{1}M_{9} + C_{17} - M_{17}\\
&\;\;\;\;\;\xrightarrow{h_{7}}
-\frac{1}{12}C_{2}C_{4}M_{1} - \frac{1}{12}C_{1}C_{4}M_{2} - \frac{1}{6}C_{4}M_{1}M_{2} + \frac{1}{6}C_{1}C_{2}M_{4} + \frac{1}{12}C_{2}M_{1}M_{4} + \frac{1}{12}C_{1}M_{2}M_{4} \\&\;\;\;\;\;\;\;\;\;\;\;\;\;- \frac{1}{2}C_{2}M_{7} + \frac{1}{2}M_{2}M_{7} + \frac{1}{2}C_{1}M_{9} - \frac{1}{2}M_{1}M_{9} + C_{17} - M_{17}\\
&\;\;\;\;\;\xrightarrow{h_i,i\leq 6} -\frac{1}{12}C_{2}C_{4}M_{1} - \frac{1}{12}C_{1}C_{4}M_{2} + \frac{1}{6}C_{1}C_{2}M_{4} + C_{17} - M_{17}=:h_{17}\\&
q_{18}\xrightarrow{h_{10}}-\frac{1}{12}C_{1}C_{3}M_{1}M_{2} + \frac{1}{8}C_{1}C_{2}M_{1}M_{3} - \frac{1}{24}C_{1}^2M_{2}M_{3} - \frac{1}{12}C_{3}C_{4}M_{1} + \frac{1}{12}C_{1}C_{6}M_{1} + \frac{1}{12}C_{6}M_{1}^2 \\&\;\;\;\;\;\;\;\;\;\;\;\;\;- \frac{1}{12}C_{1}C_{4}M_{3} - \frac{1}{6}C_{4}M_{1}M_{3} + \frac{1}{6}C_{1}C_{3}M_{4} + \frac{1}{12}C_{3}M_{1}M_{4} + \frac{1}{12}C_{1}M_{3}M_{4} - \frac{1}{12}C_{1}^2M_{6} \\&\;\;\;\;\;\;\;\;\;\;\;\;\;- \frac{1}{12}C_{1}M_{1}M_{6} - \frac{2}{3}C_{10}M_{1} + \frac{1}{3}C_{11}M_{1} + \frac{1}{2}C_{7}M_{3} - \frac{1}{2}C_{3}M_{7} + \frac{1}{2}C_{1}M_{10} \\&\;\;\;\;\;\;\;\;\;\;\;\;\;+ \frac{1}{6}M_{1}M_{10} - \frac{1}{3}M_{1}M_{11} + C_{18} - M_{18}\\
&\;\;\;\;\;\xrightarrow{h_8}\frac{1}{8}C_{1}C_{2}M_{1}M_{3} - \frac{1}{8}C_{1}^2M_{2}M_{3} - \frac{1}{12}C_{3}C_{4}M_{1} + \frac{1}{12}C_{1}C_{6}M_{1} + \frac{1}{12}C_{6}M_{1}^2 - \frac{1}{12}C_{1}C_{4}M_{3} \\&\;\;\;\;\;\;\;\;\;\;\;\;\;- \frac{1}{6}C_{4}M_{1}M_{3} + \frac{1}{6}C_{1}C_{3}M_{4} + \frac{1}{12}C_{3}M_{1}M_{4} + \frac{1}{12}C_{1}M_{3}M_{4} - \frac{1}{12}C_{1}^2M_{6} - \frac{1}{12}C_{1}M_{1}M_{6} \\&\;\;\;\;\;\;\;\;\;\;\;\;\;- \frac{2}{3}C_{10}M_{1} + \frac{1}{3}C_{11}M_{1} + C_{8}M_{2} + \frac{1}{2}C_{7}M_{3} - \frac{1}{2}C_{3}M_{7} - M_{2}M_{8} + \frac{1}{2}C_{1}M_{10} + \frac{1}{6}M_{1}M_{10} \\&\;\;\;\;\;\;\;\;\;\;\;\;\;- \frac{1}{3}M_{1}M_{11} + C_{18} - M_{18}\\
&\;\;\;\;\;\xrightarrow{h_7}-\frac{1}{12}C_{3}C_{4}M_{1} + \frac{1}{12}C_{1}C_{6}M_{1} + \frac{1}{12}C_{6}M_{1}^2 - \frac{1}{12}C_{1}C_{4}M_{3} - \frac{1}{6}C_{4}M_{1}M_{3} + \frac{1}{6}C_{1}C_{3}M_{4} \\&\;\;\;\;\;\;\;\;\;\;\;\;\;+ \frac{1}{12}C_{3}M_{1}M_{4} + \frac{1}{12}C_{1}M_{3}M_{4} - \frac{1}{12}C_{1}^2M_{6} - \frac{1}{12}C_{1}M_{1}M_{6} - \frac{2}{3}C_{10}M_{1} + \frac{1}{3}C_{11}M_{1} \\&\;\;\;\;\;\;\;\;\;\;\;\;\;+ C_{8}M_{2} - C_{7}M_{3} - \frac{1}{2}C_{3}M_{7} + \frac{3}{2}M_{3}M_{7} - M_{2}M_{8} + \frac{1}{2}C_{1}M_{10} \\&\;\;\;\;\;\;\;\;\;\;\;\;\;+ \frac{1}{6}M_{1}M_{10} - \frac{1}{3}M_{1}M_{11} + C_{18} - M_{18}\\
&\;\;\;\;\;\xrightarrow{h_i,i\leq 6} -\frac{1}{12}C_{3}C_{4}M_{1} + \frac{1}{12}C_{1}C_{6}M_{1} - \frac{1}{12}C_{1}C_{4}M_{3} + \frac{1}{6}C_{1}C_{3}M_{4} - \frac{1}{12}C_{1}^2M_{6} - \frac{2}{3} C_{10}M_{1} 
\\&\;\;\;\;\;\;\;\;\;\;\;\;\;+\frac{1}{3}C_{11}M_{1} + C_{8}M_2 - C_7M_3 + M_3M_7 - M_2M_8 + \frac{2}{3}M_1M_{10} - \frac{1}{3}M_1M_{11} + C_{18} - M_{18}=:h_{18}\\&
q_{19}\xrightarrow{h_{10}} 
\frac{1}{12}C_{1}C_{2}M_{1}M_{3} - \frac{1}{12}C_{1}^2M_{2}M_{3} - \frac{1}{12}C_{3}C_{4}M_{1} - \frac{1}{12}C_{2}C_{5}M_{1} - \frac{1}{12}C_{1}C_{5}M_{2} - \frac{1}{6}C_{5}M_{1}M_{2} \\&\;\;\;\;\;\;\;\;\;\;\;\;\;- \frac{1}{12}C_{1}C_{4}M_{3} - \frac{1}{6}C_{4}M_{1}M_{3} + \frac{1}{6}C_{1}C_{3}M_{4} + \frac{1}{12}C_{3}M_{1}M_{4} + \frac{1}{12}C_{1}M_{3}M_{4}  \\&\;\;\;\;\;\;\;\;\;\;\;\;\;+\frac{1}{6}C_{1}C_{2}M_{5} + \frac{1}{12}C_{2}M_{1}M_{5} + \frac{1}{12}C_{1}M_{2}M_{5} - \frac{1}{3}C_{10}M_{1} + \frac{1}{6}C_{11}M_{1} + \frac{1}{2}C_{8}M_{2} \\&\;\;\;\;\;\;\;\;\;\;\;\;\;+ \frac{1}{2}C_{7}M_{3} - \frac{1}{2}C_{3}M_{7} - \frac{1}{2}C_{2}M_{8} + \frac{1}{3}M_{1}M_{10} + \frac{1}{2}C_{1}M_{11} - \frac{2}{3}M_{1}M_{11} + C_{19} - M_{19}\\
&\;\;\;\;\;\xrightarrow{h_{7}}
-\frac{1}{12}C_{3}C_{4}M_{1} - \frac{1}{12}C_{2}C_{5}M_{1} - \frac{1}{12}C_{1}C_{5}M_{2} - \frac{1}{6}C_{5}M_{1}M_{2} - \frac{1}{12}C_{1}C_{4}M_{3} - \frac{1}{6}C_{4}M_{1}M_{3} \\&\;\;\;\;\;\;\;\;\;\;\;\;\;+ \frac{1}{6}C_{1}C_{3}M_{4} + \frac{1}{12}C_{3}M_{1}M_{4} + \frac{1}{12}C_{1}M_{3}M_{4} + \frac{1}{6}C_{1}C_{2}M_{5} + \frac{1}{12}C_{2}M_{1}M_{5} + \frac{1}{12}C_{1}M_{2}M_{5} \\&\;\;\;\;\;\;\;\;\;\;\;\;\;- \frac{1}{3}C_{10}M_{1} + \frac{1}{6}C_{11}M_{1} + \frac{1}{2}C_{8}M_{2} - \frac{1}{2}C_{7}M_{3} - \frac{1}{2}C_{3}M_{7} + M_{3}M_{7} - \frac{1}{2}C_{2}M_{8} \\&\;\;\;\;\;\;\;\;\;\;\;\;\;+ \frac{1}{3}M_{1}M_{10} + \frac{1}{2}C_{1}M_{11} - \frac{2}{3}M_{1}M_{11} + C_{19} - M_{19}\\
&\;\;\;\;\;\xrightarrow{h_i,i\leq 6} 
\frac{1}{12}C_{2}C_{3}M_{1}^2 - \frac{1}{12}C_{1}^2M_{2}M_{3} - \frac{1}{12}C_{3}C_{4}M_{1} - \frac{1}{12}C_{2}C_{5}M_{1} - \frac{1}{12}C_{1}C_{5}M_{2} - \frac{1}{12}C_{1}C_{4}M_{3} \\&\;\;\;\;\;\;\;\;\;\;\;\;\;+ \frac{1}{6}C_{1}C_{3}M_{4} + \frac{1}{6}C_{1}C_{2}M_{5} - \frac{1}{2}C_{11}M_{1} + \frac{1}{2}C_{8}M_{2} + \frac{1}{2}C_{7}M_{3} - \frac{1}{2}M_{3}M_{7} - \frac{1}{2}M_{2}M_{8} \\&\;\;\;\;\;\;\;\;\;\;\;\;\;+ \frac{1}{2}M_{1}M_{11} + C_{19} - M_{19}=:\tilde h_{19}\\
&\;\;\;\;\;\xrightarrow{h_{18}} 
-\frac{1}{12}C_{2}C_{5}M_{1} - \frac{1}{12}C_{1}C_{6}M_{1} - \frac{1}{12}C_{1}C_{5}M_{2} + \frac{1}{6}C_{1}C_{2}M_{5} + \frac{1}{12}C_{1}^2M_{6} + \frac{1}{3}C_{10}M_{1} \\&\;\;\;\;\;\;\;\;\;\;\;\;\;- \frac{1}{6}C_{11}M_{1} - \frac{1}{2}C_{8}M_{2} + \frac{1}{2}C_{7}M_{3} - \frac{1}{2}M_{3}M_{7} + \frac{1}{2}M_{2}M_{8} - \frac{1}{3}M_{1}M_{10} + \frac{1}{6}M_{1}M_{11} \\&\;\;\;\;\;\;\;\;\;\;\;\;\;- C_{18} + C_{19} + M_{18} - M_{19}:=h_{19}.
\end{align*}
\end{small}
Note that $\tilde h_{19}$ has $16$ terms and $h_{19}$ can not be further reduced with $17$ terms. In fact, $\{h_1,\ldots,h_{19}\}$ is \emph{reduced}. 
We compute the closure of S-polynomials and obtian 
\begin{small}
\begin{align*}
    \frac{1}{-8}(12\,M_2h_8+4\,M_1h_{11})&=
    -\frac{1}{8}C_{1}C_{2}M_{1}M_{3} + \frac{1}{8}C_{1}^2M_{2}M_{3} + C_{10}M_{1} - \frac{1}{2}C_{11}M_{1} \\&\;\;\;\;\;\;- \frac{3}{2}C_{8}M_{2} + \frac{3}{2}M_{2}M_{8} - M_{1}M_{10} + \frac{1}{2}M_{1}M_{11}\\&\xrightarrow{h_7} 
C_{10}M_{1} - \frac{1}{2}C_{11}M_{1} - \frac{3}{2}C_{8}M_{2} + \frac{3}{2}C_{7}M_{3} \\&\;\;\;\;\;\;- \frac{3}{2}M_{3}M_{7} + \frac{3}{2}M_{2}M_{8} - M_{1}M_{10} + \frac{1}{2}M_{1}M_{11}=:g_{20}
\end{align*}
\end{small}
and
\begin{small}
\begin{align*}
    2\,M_2h_{11}+M_1h_{13}&=
    -\frac{1}{4}C_{2}C_{3}M_{1}M_{2} - \frac{1}{12}C_{2}^2M_{1}M_{3} + \frac{1}{3}C_{1}C_{2}M_{2}M_{3} + C_{13}M_{1} \\&\;\;\;\;\;\;- \frac{4}{3}C_{10}M_{2} + \frac{8}{3}C_{11}M_{2} + \frac{4}{3}M_{2}M_{10} - \frac{8}{3}M_{2}M_{11} - M_{1}M_{13}
\\&\overset{h_{11}}\longrightarrow\;-\frac{1}{12}C_{2}^2M_{1}M_{3} + \frac{1}{12}C_{1}C_{2}M_{2}M_{3} + C_{13}M_{1} \\&\;\;\;\;\;\;- \frac{1}{3}C_{10}M_{2} + \frac{2}{3}C_{11}M_{2} + \frac{1}{3}M_{2}M_{10} - \frac{2}{3}M_{2}M_{11} - M_{1}M_{13}
\\&\xrightarrow{h_{9}}C_{13}M_{1} - \frac{1}{3}C_{10}M_{2} + \frac{2}{3}C_{11}M_{2} - C_{9}M_{3} \\&\;\;\;\;\;\;+ M_{3}M_{9} 
+ \frac{1}{3}M_{2}M_{10} - \frac{2}{3}M_{2}M_{11} - M_{1}M_{13}=:g_{21}
\end{align*}
\end{small}
and 
\begin{small}
\begin{align*}
    M_2h_{12}+M_1h_{14}&=
    -\frac{1}{12}C_{2}C_{3}M_{1}M_{3} + \frac{1}{12}C_{1}C_{3}M_{2}M_{3} \\&\;\;\;\;\;\;- C_{14}M_{1} + C_{12}M_{2} - M_{2}M_{12} + M_{1}M_{14}
\\&\xrightarrow{h_{11}}\frac{1}{12}C_{1}C_{3}M_{2}M_{3} - \frac{1}{12}C_{1}C_{2}M_{3}^2 - C_{14}M_{1} + C_{12}M_{2} + \frac{1}{3}C_{10}M_{3} \\&\;\;\;\;\;\;- \frac{2}{3}C_{11}M_{3} - \frac{1}{3}M_{3}M_{10} + \frac{2}{3}M_{3}M_{11} - M_{2}M_{12} + M_{1}M_{14}
\\&\xrightarrow{h_{10}}-C_{14}M_{1} + C_{12}M_{2} - \frac{1}{3}C_{10}M_{3} - \frac{1}{3}C_{11}M_{3} \\&\;\;\;\;\;\;+ \frac{1}{3}M_{3}M_{10} + \frac{1}{3}M_{3}M_{11} - M_{2}M_{12} + M_{1}M_{14}=:g_{22}.
\end{align*}
\end{small}
With this we can reduce $h_{19}$ further, i.e., 
\begin{small}
\begin{align*}
&h_{19}\xrightarrow{g_{22}} 
-\frac{1}{12}C_{2}C_{5}M_{1} 
- \frac{1}{12}C_{1}C_{6}M_{1} 
- \frac{1}{12}C_{1}C_{5}M_{2} 
+ \frac{1}{6}C_{1}C_{2}M_{5} 
\\&\;\;\;\;\;\;\;\;\;\;\;\;\;
+ \frac{1}{12}C_{1}^2M_{6} 
+\frac{1}{3}C_{10}M_1
-\frac{1}{6}C_{11}M_1
-\frac{1}{2}C_8M_2
+\frac{1}{2}C_7M_3
\\&\;\;\;\;\;\;\;\;\;\;\;\;\;
-\frac{1}{2}M_3M_7
+\frac{1}{2}M_2M_8
-\frac{1}{3}M_1M_{10}
+\frac{1}{6}M_1M_{11}
- C_{18} + C_{19} + M_{18} - M_{19}  =: g_{19}
 \end{align*}
\end{small}
and thus obtain an output of \Cref{algo:Buchberger},  
$$G^{(19)}=\operatorname{Buchberger}(q_1,\dots,q_{19})=\{h_1,\dots,h_{18}\}\cup\{g_{19},\dots,g_{22}\}$$
due to \Cref{cor:rj_def}.
Note that the two cases 
\begin{small}
\begin{align*}
&
\tilde h_{11}=\operatorname{rNF}(q_{11},G^{(9)})=:s_{11}\\&
\tilde h_{19}\xrightarrow{g_{22}} -\frac{1}{12}C_{3}C_{4}M_{1} - \frac{1}{12}C_{2}C_{5}M_{1} - \frac{1}{12}C_{1}C_{5}M_{2} \\&\;\;\;\;\;\;\;\;\;\;\;\;\;- \frac{1}{12}C_{1}C_{4}M_{3} + \frac{1}{6}C_{1}C_{3}M_{4} + \frac{1}{6}C_{1}C_{2}M_{5} + C_{19} - M_{19}=\operatorname{rNF}(q_{19},G^{(17)}) =: s_{19}
 \end{align*}
\end{small}
provide relations with fewer summands as $\operatorname{rNF}(q_{j},G^{(j-1)})$.
For all other $j\in\{1,\ldots,18\}\setminus\{11\}$ we set 
$s_j:=\operatorname{rNF}(q_{j},G^{(j-1)})$. For $j\leq 20$, we set
\begin{small}
\begin{align*}
    s_{20}:=\operatorname{rNF}(q_{20},G^{(19)}) &=-\frac{1}{12}C_{3}C_{5}M_{1} - \frac{1}{12}C_{1}C_{5}M_{3} + \frac{1}{6}C_{1}C_{3}M_{5} + C_{20} - M_{20}\\
    s_{21}:=\operatorname{rNF}(q_{21},G^{(17)}) &=\frac{1}{6}C_{3}C_{4}M_{1} - \frac{1}{6}C_{2}C_{5}M_{1} + \frac{1}{12}C_{1}C_{5}M_{2} - \frac{1}{12}C_{1}C_{4}M_{3} \\&\;\;\;\;\;\;- \frac{1}{12}C_{1}C_{3}M_{4} + \frac{1}{12}C_{1}C_{2}M_{5} + C_{21} - M_{21}\\
     s_{22}:=\operatorname{rNF}(q_{22},G^{(21)}) &=\frac{1}{12}C_{2}C_{4}M_{2} - \frac{1}{12}C_{2}^2M_{4} + C_{22} - M_{22}\\
     s_{23}:=\operatorname{rNF}(q_{23},G^{(22)}) &=-\frac{1}{12}C_{2}C_{6}M_{1} + \frac{1}{6}C_{3}C_{4}M_{2} + \frac{1}{6}C_{1}C_{6}M_{2} - \frac{1}{12}C_{2}C_{4}M_{3} \\&\;\;\;\;\;\;- \frac{1}{12}C_{2}C_{3}M_{4} - \frac{1}{12}C_{1}C_{2}M_{6}  + C_{23} - M_{23}\\
     s_{24}:=\operatorname{rNF}(q_{24},G^{(22)}) &=-\frac{1}{6}C_{2}C_{6}M_{1} + \frac{1}{4}C_{3}C_{4}M_{2} + \frac{1}{12}C_{1}C_{6}M_{2} - \frac{1}{4}C_{2}C_{3}M_{4} \\&\;\;\;\;\;\;+ \frac{1}{12}C_{1}C_{2}M_{6}+ C_{24} - M_{24}\\
     s_{25}:=\operatorname{rNF}(q_{25},G^{(24)}) &=-\frac{1}{12}C_{3}C_{6}M_{1} + \frac{1}{12}C_{3}C_{4}M_{3} - \frac{1}{12}C_{1}C_{6}M_{3} - \frac{1}{12}C_{3}^2M_{4}  \\&\;\;\;\;\;\;+ \frac{1}{6}C_{1}C_{3}M_{6}+ C_{25} - M_{25}\\
     s_{26}:=\operatorname{rNF}(q_{26},G^{(23)}) &=\frac{1}{12}C_{3}C_{4}M_{2} + \frac{1}{12}C_{2}C_{5}M_{2} + \frac{1}{12}C_{2}C_{4}M_{3}  \\&\;\;\;\;\;\;- \frac{1}{6}C_{2}C_{3}M_{4} - \frac{1}{12}C_{2}^2M_{5}+ C_{26} - M_{26}\\
     s_{27}:=\operatorname{rNF}(q_{27},G^{(26)}) &=\frac{1}{6}C_{3}C_{5}M_{2} - \frac{1}{12}C_{2}C_{5}M_{3} + \frac{1}{4}C_{1}C_{6}M_{3} \\&\;\;\;\;\;\;- \frac{1}{12}C_{2}C_{3}M_{5} - \frac{1}{4}C_{1}C_{3}M_{6} - 2C_{25} + C_{27} + 2M_{25} - M_{27}\\
     s_{28}:=\operatorname{rNF}(q_{28},G^{(26)}) &=\frac{1}{12}C_{3}C_{5}M_{2} + \frac{1}{12}C_{3}C_{4}M_{3} + \frac{1}{12}C_{2}C_{5}M_{3} - \frac{1}{12}C_{3}^2M_{4} \\&\;\;\;\;\;\;- \frac{1}{6}C_{2}C_{3}M_{5} + C_{28} - M_{28}\\
     s_{29}:=\operatorname{rNF}(q_{29},G^{(28)}) &=\frac{1}{12}C_{3}C_{5}M_{3} - \frac{1}{12}C_{3}^2M_{5} + C_{29} - M_{29}\\
     s_{30}:=\operatorname{rNF}(q_{30},G^{(29)}) &=\frac{1}{12}C_{2}C_{6}M_{2} - \frac{1}{12}C_{2}^2M_{6} + C_{30} - M_{30}\\
     s_{31}:=\operatorname{rNF}(q_{31},G^{(30)}) &=-\frac{1}{12}C_{3}C_{6}M_{2} - \frac{1}{12}C_{2}C_{6}M_{3} + \frac{1}{6}C_{2}C_{3}M_{6} + C_{31} - M_{31}\\
     s_{32}:=\operatorname{rNF}(q_{32},G^{(31)}) &=\frac{1}{12}C_{3}C_{6}M_{3} - \frac{1}{12}C_{3}^2M_{6} + C_{32} - M_{32}. 
\end{align*}
\end{small}
\Cref{table:r_j_L4_d3} shows the number of sumands for the correction terms $r=s-M+C$. 
\begin{table}
\begin{center}
\begin{small}
    \begin{tabular}{ c| c c c c c c c c c c c c c c c c}
  $j$ & $1$ & $2$ & $3$ & $4$ & $5$ & $6$ &$7$& $8$ & $9$ & $10$ & $11$ & $12$ & $13$ &$14$& $15$ &$16$\\\hline
  $Q_{L,d,j}$ & $0$ &$ 0$ & $ 0$ & $ 0$ & $ 0$ & $ 0$ & $ 2$ & $ 2$ & $ 2$ & $3$  & $3$ & $2$ & $2$ & $2$& $2$ & $ 2$\\\\
  $j$ & $17$ & $18$ & $19$ & $20$ & $21$ & $22$ &$23$& $24$ & $25$ & $26$ & $27$ & $28$ & $29$ &$30$& $31$ &$32$\\\hline
  $Q_{L,d,j}$ & $3$ &$5$ & $6$ & $3$ & $6$ & $2$ & $6$ & $ 5$ & $5$ & $5$ & $7$ & $5$ & $2$ & $2$& $3$ & $2$
\end{tabular}
\end{small}
\caption{Number of sumands $Q_{L,d,j}$ in $r_j$ for $d=3$ and $L=4$.}
\label{table:r_j_L4_d3}
\end{center}
\end{table}

We provide explicit verifications for
\begin{small}
\begin{align*}
q_{25}&=\frac{1}{24}C_{3}^2M_{1}M_{2} - \frac{1}{12}C_{2}C_{3}M_{1}M_{3} + \frac{1}{12}C_{1}C_{3}M_{2}M_{3} - \frac{1}{24}C_{1}C_{2}M_{3}^2 - \frac{1}{12}C_{3}C_{6}M_{1} \\&\;\;\;\;\;\;\;\;\;\;\;\;\;+ \frac{1}{12}C_{3}C_{4}M_{3} - \frac{1}{12}C_{1}C_{6}M_{3} - \frac{1}{6}C_{6}M_{1}M_{3} + \frac{1}{12}C_{4}M_{3}^2 - \frac{1}{12}C_{3}^2M_{4} \\&\;\;\;\;\;\;\;\;\;\;\;\;\;- \frac{1}{12}C_{3}M_{3}M_{4} + \frac{1}{6}C_{1}C_{3}M_{6} + \frac{1}{12}C_{3}M_{1}M_{6} + \frac{1}{12}C_{1}M_{3}M_{6} - \frac{1}{2}C_{14}M_{1} \\&\;\;\;\;\;\;\;\;\;\;\;\;\;+ \frac{1}{2}C_{10}M_{3} - \frac{1}{2}C_{3}M_{10} + \frac{1}{2}C_{1}M_{14} + C_{25} - M_{25} \\&\xrightarrow{h_{12}}
-\frac{1}{12}C_{2}C_{3}M_{1}M_{3} + \frac{1}{8}C_{1}C_{3}M_{2}M_{3} - \frac{1}{24}C_{1}C_{2}M_{3}^2 - \frac{1}{12}C_{3}C_{6}M_{1} + \frac{1}{12}C_{3}C_{4}M_{3} \\&\;\;\;\;\;\;\;\;\;\;\;\;\;- \frac{1}{12}C_{1}C_{6}M_{3} - \frac{1}{6}C_{6}M_{1}M_{3} + \frac{1}{12}C_{4}M_{3}^2 - \frac{1}{12}C_{3}^2M_{4} - \frac{1}{12}C_{3}M_{3}M_{4} \\&\;\;\;\;\;\;\;\;\;\;\;\;\;+ \frac{1}{6}C_{1}C_{3}M_{6} + \frac{1}{12}C_{3}M_{1}M_{6} + \frac{1}{12}C_{1}M_{3}M_{6} - \frac{1}{2}C_{14}M_{1} + \frac{1}{2}C_{12}M_{2} + \frac{1}{2}C_{10}M_{3} \\&\;\;\;\;\;\;\;\;\;\;\;\;\;- \frac{1}{2}C_{3}M_{10} - \frac{1}{2}M_{2}M_{12} + \frac{1}{2}C_{1}M_{14} + C_{25} - M_{25}
\\&\xrightarrow{h_{11}}
\frac{1}{8}C_{1}C_{3}M_{2}M_{3} - \frac{1}{8}C_{1}C_{2}M_{3}^2 - \frac{1}{12}C_{3}C_{6}M_{1} + \frac{1}{12}C_{3}C_{4}M_{3} - \frac{1}{12}C_{1}C_{6}M_{3} \\&\;\;\;\;\;\;\;\;\;\;\;\;\;- \frac{1}{6}C_{6}M_{1}M_{3} + \frac{1}{12}C_{4}M_{3}^2 - \frac{1}{12}C_{3}^2M_{4} - \frac{1}{12}C_{3}M_{3}M_{4} + \frac{1}{6}C_{1}C_{3}M_{6} \\&\;\;\;\;\;\;\;\;\;\;\;\;\;+ \frac{1}{12}C_{3}M_{1}M_{6} + \frac{1}{12}C_{1}M_{3}M_{6} - \frac{1}{2}C_{14}M_{1} + \frac{1}{2}C_{12}M_{2} + \frac{5}{6}C_{10}M_{3} - \frac{2}{3}C_{11}M_{3} \\&\;\;\;\;\;\;\;\;\;\;\;\;\;- \frac{1}{2}C_{3}M_{10} - \frac{1}{3}M_{3}M_{10} + \frac{2}{3}M_{3}M_{11} - \frac{1}{2}M_{2}M_{12} + \frac{1}{2}C_{1}M_{14} + C_{25} - M_{25}
\\&\xrightarrow{h_{10}}
-\frac{1}{12}C_{3}C_{6}M_{1} + \frac{1}{12}C_{3}C_{4}M_{3} - \frac{1}{12}C_{1}C_{6}M_{3} - \frac{1}{6}C_{6}M_{1}M_{3} + \frac{1}{12}C_{4}M_{3}^2 \\&\;\;\;\;\;\;\;\;\;\;\;\;\;- \frac{1}{12}C_{3}^2M_{4} - \frac{1}{12}C_{3}M_{3}M_{4} + \frac{1}{6}C_{1}C_{3}M_{6} + \frac{1}{12}C_{3}M_{1}M_{6} + \frac{1}{12}C_{1}M_{3}M_{6} \\&\;\;\;\;\;\;\;\;\;\;\;\;\;- \frac{1}{2}C_{14}M_{1} + \frac{1}{2}C_{12}M_{2} - \frac{1}{6}C_{10}M_{3} - \frac{1}{6}C_{11}M_{3} - \frac{1}{2}C_{3}M_{10} \\&\;\;\;\;\;\;\;\;\;\;\;\;\;+ \frac{2}{3}M_{3}M_{10} + \frac{1}{6}M_{3}M_{11} - \frac{1}{2}M_{2}M_{12} + \frac{1}{2}C_{1}M_{14} + C_{25} - M_{25}
\\&\xrightarrow{h_{i},i\leq 6}
-\frac{1}{12}C_{3}C_{6}M_{1} + \frac{1}{12}C_{3}C_{4}M_{3} -\frac{1}{12}C_{1}C_{6}M_{3} -\frac{1}{12}C_{3}^2M_{4} + \frac{1}{6}C_{1}C_{3}M_{6} \\&\;\;\;\;\;\;\;\;\;\;\;\;\;-\frac{1}{2}C_{14}M_{1} + \frac{1}{2}C_{12}M_{2} - \frac{1}{6}C_{10}M_{3} - \frac{1}{6}C_{11}M_{3} - \frac{1}{2}C_{3}M_{10} + \frac{2}{3}M_{3}M_{10} \\&\;\;\;\;\;\;\;\;\;\;\;\;\;+ \frac{1}{6}M_{3}M_{11} - \frac{1}{2}M_{2}M_{12} + \frac{1}{2}C_{1}M_{14} + C_{25} - M_{25}
\\&\xrightarrow{g_{22}}-\frac{1}{2}C_{3}M_{10} + \frac{1}{2}M_{3}M_{10} + \frac{1}{2}C_{1}M_{14} - \frac{1}{2}M_{1}M_{14} + C_{25} - M_{25}
\\&\xrightarrow{h_{i},i\leq 6}-\frac{1}{12}C_{3}C_{6}M_{1} + \frac{1}{12}C_{3}C_{4}M_{3} - \frac{1}{12}C_{1}C_{6}M_{3} - \frac{1}{12}C_{3}^2M_{4} \\&\;\;\;\;\;\;\;\;\;\;\;\;\;+ \frac{1}{6}C_{1}C_{3}M_{6} + C_{25} - M_{25} = s_{25}
 \end{align*}
\end{small}
and the the ``largest'' relation with $7$ terms, 
   \begin{small}
\begin{align*}
q_{27}&=-\frac{1}{8}C_{2}C_{3}M_{1}M_{3} + \frac{1}{8}C_{1}C_{3}M_{2}M_{3} - \frac{1}{6}C_{3}C_{6}M_{1} + \frac{1}{6}C_{3}C_{5}M_{2} + \frac{1}{6}C_{3}C_{4}M_{3} \\&\;\;\;\;\;\;\;\;\;\;\;\;\;- \frac{1}{12}C_{2}C_{5}M_{3} + \frac{1}{12}C_{1}C_{6}M_{3} - \frac{1}{12}C_{6}M_{1}M_{3} + \frac{1}{12}C_{5}M_{2}M_{3} + \frac{1}{6}C_{4}M_{3}^2 \\&\;\;\;\;\;\;\;\;\;\;\;\;\;- \frac{1}{6}C_{3}^2M_{4} - \frac{1}{6}C_{3}M_{3}M_{4} - \frac{1}{12}C_{2}C_{3}M_{5} + \frac{1}{12}C_{3}M_{2}M_{5} - \frac{1}{6}C_{2}M_{3}M_{5} \\&\;\;\;\;\;\;\;\;\;\;\;\;\;+ \frac{1}{12}C_{1}C_{3}M_{6} - \frac{1}{12}C_{3}M_{1}M_{6} + \frac{1}{6}C_{1}M_{3}M_{6} + \frac{1}{2}C_{10}M_{3} + \frac{1}{2}C_{11}M_{3} \\&\;\;\;\;\;\;\;\;\;\;\;\;\;- \frac{1}{2}C_{6}M_{5} + \frac{1}{2}C_{5}M_{6} - \frac{1}{2}C_{3}M_{10} - \frac{1}{2}C_{3}M_{11} + C_{27} - M_{27}
\\&\xrightarrow{h_{11}}
\frac{1}{8}C_{1}C_{3}M_{2}M_{3} - \frac{1}{8}C_{1}C_{2}M_{3}^2 - \frac{1}{6}C_{3}C_{6}M_{1} + \frac{1}{6}C_{3}C_{5}M_{2} + \frac{1}{6}C_{3}C_{4}M_{3} \\&\;\;\;\;\;\;\;\;\;\;\;\;\;- \frac{1}{12}C_{2}C_{5}M_{3} + \frac{1}{12}C_{1}C_{6}M_{3} - \frac{1}{12}C_{6}M_{1}M_{3} + \frac{1}{12}C_{5}M_{2}M_{3} + \frac{1}{6}C_{4}M_{3}^2 \\&\;\;\;\;\;\;\;\;\;\;\;\;\;- \frac{1}{6}C_{3}^2M_{4} - \frac{1}{6}C_{3}M_{3}M_{4} - \frac{1}{12}C_{2}C_{3}M_{5} + \frac{1}{12}C_{3}M_{2}M_{5} - \frac{1}{6}C_{2}M_{3}M_{5} \\&\;\;\;\;\;\;\;\;\;\;\;\;\;+ \frac{1}{12}C_{1}C_{3}M_{6} - \frac{1}{12}C_{3}M_{1}M_{6} + \frac{1}{6}C_{1}M_{3}M_{6} + C_{10}M_{3} - \frac{1}{2}C_{11}M_{3} \\&\;\;\;\;\;\;\;\;\;\;\;\;\;- \frac{1}{2}C_{6}M_{5} + \frac{1}{2}C_{5}M_{6} - \frac{1}{2}C_{3}M_{10} - \frac{1}{2}M_{3}M_{10} \\&\;\;\;\;\;\;\;\;\;\;\;\;\;- \frac{1}{2}C_{3}M_{11} + M_{3}M_{11} + C_{27} - M_{27}
\\&\xrightarrow{h_{10}}
-\frac{1}{6}C_{3}C_{6}M_{1} + \frac{1}{6}C_{3}C_{5}M_{2} + \frac{1}{6}C_{3}C_{4}M_{3} - \frac{1}{12}C_{2}C_{5}M_{3} + \frac{1}{12}C_{1}C_{6}M_{3} \\&\;\;\;\;\;\;\;\;\;\;\;\;\;- \frac{1}{12}C_{6}M_{1}M_{3} + \frac{1}{12}C_{5}M_{2}M_{3} + \frac{1}{6}C_{4}M_{3}^2 - \frac{1}{6}C_{3}^2M_{4} - \frac{1}{6}C_{3}M_{3}M_{4} \\&\;\;\;\;\;\;\;\;\;\;\;\;\;- \frac{1}{12}C_{2}C_{3}M_{5} + \frac{1}{12}C_{3}M_{2}M_{5} - \frac{1}{6}C_{2}M_{3}M_{5} + \frac{1}{12}C_{1}C_{3}M_{6} - \frac{1}{12}C_{3}M_{1}M_{6} \\&\;\;\;\;\;\;\;\;\;\;\;\;\;+ \frac{1}{6}C_{1}M_{3}M_{6} - \frac{1}{2}C_{6}M_{5} + \frac{1}{2}C_{5}M_{6} - \frac{1}{2}C_{3}M_{10} + \frac{1}{2}M_{3}M_{10} \\&\;\;\;\;\;\;\;\;\;\;\;\;\;- \frac{1}{2}C_{3}M_{11} + \frac{1}{2}M_{3}M_{11} + C_{27} - M_{27}
\\&\xrightarrow{s_{25}}
\frac{1}{6}C_{3}C_{5}M_{2} - \frac{1}{12}C_{2}C_{5}M_{3} + \frac{1}{4}C_{1}C_{6}M_{3} - \frac{1}{12}C_{6}M_{1}M_{3} + \frac{1}{12}C_{5}M_{2}M_{3} \\&\;\;\;\;\;\;\;\;\;\;\;\;\;+ \frac{1}{6}C_{4}M_{3}^2 - \frac{1}{6}C_{3}M_{3}M_{4} - \frac{1}{12}C_{2}C_{3}M_{5} + \frac{1}{12}C_{3}M_{2}M_{5} - \frac{1}{6}C_{2}M_{3}M_{5} \\&\;\;\;\;\;\;\;\;\;\;\;\;\;- \frac{1}{4}C_{1}C_{3}M_{6} - \frac{1}{12}C_{3}M_{1}M_{6} + \frac{1}{6}C_{1}M_{3}M_{6} - \frac{1}{2}C_{6}M_{5} + \frac{1}{2}C_{5}M_{6} \\&\;\;\;\;\;\;\;\;\;\;\;\;\;- \frac{1}{2}C_{3}M_{10} + \frac{1}{2}M_{3}M_{10} - \frac{1}{2}C_{3}M_{11} + \frac{1}{2}M_{3}M_{11} - 2C_{25} + C_{27} + 2M_{25} - M_{27}
\\&\xrightarrow{h_{i},i\leq 6}
\frac{1}{6}C_{3}C_{5}M_{2} - \frac{1}{12}C_{2}C_{5}M_{3} + \frac{1}{4}C_{1}C_{6}M_{3} - \frac{1}{12}C_{2}C_{3}M_{5} \\&\;\;\;\;\;\;\;\;\;\;\;\;\;- \frac{1}{4}C_{1}C_{3}M_{6} - 2C_{25} + C_{27} + 2M_{25} - M_{27}=s_{27}.
\end{align*}
\end{small}
 Let $\tilde r_j$ denote the correction terms due to  \Cref{thm:aBCH}. Then $r_j=\tilde r_j$ for all $j\not=27$ but 
\begin{align*}{\tilde r}_{27}=-\frac{1}{6}C_{3}C_{6}M_{1} 
&+ \frac{1}{6}C_{3}C_{5}M_{2} + \frac{1}{6}C_{3}C_{4}M_{3} - \frac{1}{12}C_{2}C_{5}M_{3} \\ 
&+ \frac{1}{12}C_{1}C_{6}M_{3}- \frac{1}{6}C_{3}^2M_{4} - \frac{1}{12}C_{2}C_{3}M_{5} + \frac{1}{12}C_{1}C_{3}M_{6}\end{align*}
has $8$ terms. 
In fact, $\operatorname{rNF}(r_{27}-{\tilde r}_{27},G^{(26)})=0$
with $r_{27}={\tilde r}_{27}+2\,s_{25}$.

\section{Expansion in the ambient space for order 5}
Using \Cref{cor:abch_ambient}, we can find the remainder for $L=5$ as
\begin{align*}
 &R_5(B,C)  =  C_5\\
 &+\underbrace{\frac{1}{12}([[b_1,c_2],c_2] + [[b_3,c_1],c_1]  + [[b_2,c_2],c_1] + [[b_2,c_1],c_2]+ [[b_1,c_3],c_1] + [[b_1,c_1],c_3])}_{III} \\
&-\underbrace{\frac{1}{720}[b_1,[b_1,[[b_1,c_1],c_1]]]+\frac{1}{360}[[b_1,[b_1,c_1]],[b_1,c_1]]}_{IV} \\
&+
\underbrace{\frac{1}{180}[b_1,[[[b_1,c_1],c_1],c_1]]+\frac{1}{120}[[b_1,c_1],[[b_1,c_1],c_1]]}_{V}-\underbrace{\frac{1}{720}[[[[b_1,c_1],c_1],c_1],c_1]}_{VI}.
\end{align*}
The term $III$ combines the structure of  \eqref{eq:R4_term1} and \eqref{eq:R4_term2}, hence has $2(3+6) = 18$ monomials respectively.
The term $VI$ has $5$ monomials (where the tensor product is omitted for simplicity)
\[
VI = [[[[b_1,c_1],c_1],c_1],c_1] = b_1   c_1^{4} - 4 c_1 b_1 c_1^{3} + 6 c_1^{2}  b_1  c_1^{2} - 
4 c_1^{3}  b_1  c_1 + c_1^{4}  b_1.
\]
The term $V$ has $10$ monomials
\[
V = \frac{1}{360} (2b^2 c^3 + 2 c^3 b^2+ 2 bc^3b - 3 bc^2bc - 3cbc^2b - 3 bcbc^2 - 3 c^2bcb - 3 c^2b^2c - 3cb^2c^2 + 12 cbcbc)
\]
where we omitted the subscript $\cdot_1$; the term $IV$ has also 10 monomials 
\[
IV = \frac{1}{720}(4cb^3c - b^3 c^2 - c^2b^3 -  b^2 c^2 b -  b c^2b^2 + 4 b^2cbc + 4 cbcb^2  -6cb^2cb - 6bcb^2c + 4bcbcb).
\]
That gives $43 = 18+10+10+5$ monomials in total for order $5$ (again agreeing with the upper bound in \Cref{table:r_j}).

\section{Explicit form of projection for order 4}
\begin{example}
The combinatorial formula for $\pi_1$ at order $4$ it gives:
\begin{align*}
&(\pi_{1,4} (a))  = \frac{1}{4}  \left(a - \rev(a)\right) \\
- \frac{1}{12} \Big(\phantom{+}&\mathsf{permute}_{(1243)}(a) + \mathsf{permute}_{(1324)}(a) + \mathsf{permute}_{(1342)}(a) + \mathsf{permute}_{(1423)} (a)\\
 +&  \mathsf{permute}_{(2134)}(a) + \mathsf{permute}_{(2314)}(a) + \mathsf{permute}_{(2341)}(a)+ \mathsf{permute}_{(2413)} (a) \\
 +& \mathsf{permute}_{(3124)}(a)   + \mathsf{permute}_{(3412)}(a) + \mathsf{permute}_{(4123)}(a)\Big)  \\
+ \frac{1}{12} \Big(\phantom{+}&\mathsf{permute}_{(3421)}(a) + \mathsf{permute}_{(4321)}(a) + \mathsf{permute}_{(2431)}(a) + \mathsf{permute}_{(3241)}(a) \\
 +&  \mathsf{permute}_{(4312)}(a) + \mathsf{permute}_{(4132)}(a) + \mathsf{permute}_{(1432)}(a)+ \mathsf{permute}_{(3142)} (a) \\
 + &\mathsf{permute}_{(4213)}(a)   + \mathsf{permute}_{(2143)}(a) + \mathsf{permute}_{(3214)}(a)\Big)  \\
\end{align*}
\end{example}

\end{document}